\documentclass[11pt,reqno]{amsart}
\usepackage{amsfonts}
\usepackage{amsmath}
\usepackage{amssymb}
\usepackage{amsthm,bbm,bm}
\usepackage{multicol}
\usepackage{stmaryrd}
\usepackage{cite}
\usepackage{epsfig}
\usepackage{color}
\usepackage{graphics}
\usepackage{graphicx}
\usepackage{hyperref}
\usepackage{subfigure}
\usepackage{epstopdf}
\usepackage{multicol,graphics}
\newcommand\bes{\begin{eqnarray}}
\newcommand\ees{\end{eqnarray}}

\newtheorem{theorem}{Theorem}[section]
\newtheorem{lemma}[theorem]{Lemma}
\newtheorem{corollary}[theorem]{Corollary}

\newtheorem{definition}[theorem]{Definition}

\newtheorem{proposition}[theorem]{Proposition}
\newtheorem{hypothesis}[theorem]{Hypothesis}
\numberwithin{equation}{section}
\allowdisplaybreaks

\begin{document}
\title[Propagation Dynamics  for Multidimensional Nonlocal Diffusion Equations]{\textbf{Propagation  Dynamics for Multidimensional Nonlocal Diffusion Equations: A General Freidlin-G\"artner Formula}}

\author[Guo, Li and Liu]{Hongjun Guo$^{1}$, Wan-Tong Li$^{2}$, Biao Liu$^{2,*}$}

\thanks{\hspace{-.6cm}
$^1$ School of Mathematical Sciences, Key Laboratory of Intelligent Computing and Applications (Ministry of Education), Institute for Advanced Study, Tongji University, Shanghai 200000 P.R. China.\\
$^2$ School of Mathematics and Statistics, Lanzhou University, Lanzhou, Gansu 730000, P.R. China.\\
$^*${\sf Corresponding author} (lbiao2023@lzu.edu.cn).}

\date{\today}

\begin{abstract}

In this paper, we develop a unified geometric framework for the propagation dynamics of multidimensional nonlocal Fisher--KPP equations with asymmetric dispersal. In dimensions $N\geqslant2$,
asymmetry leads to a  multidimensional propagation geometry determined by the full family of minimal  wave speeds, and the intrinsic spreading set may even fail to contain the origin. We show that the mean dispersal drift
$\mathbf p=\int_{\mathbb R^N}x\mathcal J(x)dx$
provides a natural center for the spreading geometry, yielding a new drift-centered radial description.

For unbounded initial supports, we prove that the spreading set $\mathcal W$ is given by the Minkowski-sum formula.
Although  $\mathcal W$  need not be convex, it remains star-shaped with respect to $\mathbf p$ and admits a Freidlin--G\"artner-type variational characterization. We further establish locally truncated Hausdorff convergence under a hypothesis on directions, and a global $o(t)$-Hausdorff approximation under a finite-thickness assumption on the initial support. Thus, our results extend the classical Freidlin--G\"artner picture
to genuinely biased multidimensional nonlocal propagation.


\end{abstract}
\maketitle

\section{Introduction}
In this paper, we study the long-time dynamics  of solutions to the nonlocal diffusion equation
\begin{equation}\label{equation}
u_t = \mathcal{J}\star u - u + f(u),\ t>0,\ x \in \mathbb{R}^N,
\end{equation}
 with  the Heaviside-type initial condition
\begin{eqnarray}\label{U}
u_0(x) =\mathbbm{1}_U= \left\{\begin{array}{lll}
1, & x \in U, \\
0, & x \in \mathbb{R}^N \setminus U,
\end{array}
\right.
\end{eqnarray}
where $U$ is a  measurable subset of $\mathbb{R}^N$  (possibly unbounded),  referred to as the initial support. Here,
$$\mathcal{J}\star u(t,x)- u(t,x)=\int_{\mathbb{R}^N}\mathcal{J}(x-y)u(t,y)dy-u(t,x).$$
 The kernel $\mathcal J\geqslant 0$ is a probability density, but crucially, we do \textbf{not} assume it is symmetric.
The dispersal kernel $\mathcal J$ satisfies that there exist $r_0,\kappa_0,\mu>0$ such that
\begin{itemize}
\item[(J):] $\mathcal J\in L^1(\mathbb R^N)\cap L^\infty(\mathbb R^N)$,
$\int_{\mathbb R^N}\mathcal J(x)\,dx=1$,
$\mathcal J\geqslant0$ a.e. in $\mathbb R^N$,
$\mathcal J(x)\geqslant\kappa_0$
 for a.e. $x\in B_{r_0}(0)$ and
$\int_{\mathbb R^N}
\mathcal J(x)e^{\mu|x|}\,dx<+\infty$.
\end{itemize}
 We consider a nonlinearity  $f$ of KPP type, satisfying
 $$ \text{(F)}:f\in C^1([0,1])\text{ is  a concave function such that  } f(0)=f(1)=0 \text{ and } f'(1)<0<f'(0).$$
The solution is understood in the sense that $\|u(t,\cdot)-u_0(\cdot)\|_{\infty}\rightarrow 0$ as $t\rightarrow 0^+$ and the maximum principle implies that $0\leqslant u(t,x)\leqslant 1$, see \cite{FKM}. We are interested in the long-time behavior of the solution $u(t,x)$.

For the classical reaction-diffusion model
\begin{equation}\label{local-equation}
u_t = \Delta u + f(u),\  t>0,\ x\in\mathbb{R}^N,
\end{equation}
the theory of propagation is built on two pillars\cite{F1937,KPP1937,AW1978}: the existence of traveling front solutions $u(t,x)=\phi(x\cdot e - ct)$
  with speed $c$ and direction $e\in\mathbb{S}^{N-1}$,  and the concept of the asymptotic spreading speed.   In 1937, Fisher \cite{F1937} and Kolmogorov et al. \cite{KPP1937} proved that for $N=1$, \eqref{local-equation} with (F) has traveling front solutions $\phi(x - ct)$ with  wave speed $c$  if and only if $c\geqslant c^*:=2\sqrt{f'(0)}$. Later, Aronson and Weinberger \cite{AW1978} established that  the asymptotic spreading speeds of solutions to \eqref{local-equation} with compactly supported initial conditions are equal to $c^*$ for all directions in arbitrary space dimension. Here, we say that for $e\in\mathbb{S}^{N-1}$, a positive quantity $s(e)$ is the (asymptotic) spreading speed of a solution $u(t,x)$ in the direction $e$ if
\begin{eqnarray}\label{sspeed}
\left\{\begin{array}{lll}
u(t,cte)\rightarrow 1, && \hbox{ as $t\rightarrow +\infty$ for every } 0< c<s(e),\\
u(t,cte)\rightarrow 0, && \hbox{ as $t\rightarrow +\infty$ for every } c>s(e).
\end{array}
\right.
\end{eqnarray}In spatially heterogeneous media, the traveling front solutions and their speeds may depend on  the direction, and so do the spreading speeds even for the solution with compactly supported initial condition. In particular, for spatially periodic equations with KPP-type reaction terms and suitable conditions on coefficients
  \begin{equation}\label{periodic}
  u_t=\nabla\cdot(A(x)\nabla u)+q(x)\cdot \nabla u+f(x,u),
  \end{equation}
  it is known from \cite{BH02} that there is a minimal speed $c_e^*>0$ such that the equation admits pulsating traveling fronts $\Phi_e(x\cdot e-c_e t,x)$ in every direction $e\in\mathbb{S}^{N-1}$ if and only if $c_e\geqslant c_e^*$ and that
 the minimal speed $c_e^*$ generally depends on the direction $e$. The spreading speed $s(e)$ of the solution with compactly supported initial condition is then given by
\begin{equation}\label{FG}
s(e)=\inf_{\xi\cdot e>0}\frac{c^*_\xi}{\xi\cdot e}>0,\ \xi\in\mathbb S^{N-1},
\end{equation}
see \cite{F1984,FG1979,R2017,BFN2005,GHR2025,DuLiShen2022,DuLiXin2026,LZ2007}. This formula is known as Freidlin-G\"artner formula, which was discovered by  Freidlin and G\"artner \cite{F1984,FG1979} by using large deviation probabilistic techniques.

The classical Freidlin-G\"artner framework relies essentially on the positivity of the  minimal wave speeds. However, as one will see in the sequel, the minimal speed of traveling front solutions for the nonlocal problem \eqref{equation} may be negative in some directions, see \cite{XLR2021,DFN} for the one-dimensional case. To avoid using the positivity of the spreading speed, we focus on the spreading set instead.

\begin{definition}\label{spreadingset} We say that a bounded solution $u(t,x)$ to \eqref{equation} with initial data $u_0(x)$ admits a nonempty \textit{spreading set}  $\mathcal{W} \subset \mathbb{R}^N$ if $\mathcal{W}=\operatorname{int}(\overline {\mathcal W})$ and
\begin{eqnarray}\label{definitionw}
\left\{\begin{array}{lll}
\lim \limits_{t \rightarrow+\infty}\Big(\displaystyle\inf_{x \in C} u(t, t x)\Big)=1 && \text { for any nonempty compact set } C \subset \operatorname{int}\mathcal{W}, \\
\lim \limits_{t \rightarrow+\infty}\Big(\displaystyle \sup_{x \in C} u(t, t x)\Big)=0 && \text { for any nonempty compact set } C \subset \mathbb{R}^N \backslash \overline{\mathcal{W}} .
\end{array}
\right.
\end{eqnarray}
\end{definition}

The condition $\mathcal{W}=\operatorname{int}(\overline {\mathcal W})$  in Definition~\ref{spreadingset} guarantees  the uniqueness of the spreading set. To see this, suppose $\mathcal W'$ is another spreading set. Since $\mathcal W$ is open, for each $x\in\mathcal W$, there is a compact $C$ with $x\in C\subset \mathcal W$. Property \eqref{definitionw} then implies  $x\not\in\mathbb R^{N}\backslash\overline{\mathcal W'}$, hence $\mathcal W\subset \overline{\mathcal W'}$. Because  $\mathcal W$ is open,  it follows that $\mathcal W \subset \operatorname{int}(\overline{\mathcal W'})=\mathcal W'$, and vice versa.

 Clearly, the spreading set of \eqref{local-equation} with compactly supported initial condition is the ball with radius $c^*$ centered at $0$. Equation \eqref{periodic} with compactly supported initial condition admits a spreading set, which we call the  \textbf{intrinsic spreading set} $\mathcal S$ (all equations \eqref{equation}, \eqref{local-equation} and \eqref{periodic} share the same type of intrinsic spreading set)
\begin{equation}\label{mathcals}
\mathcal{S}:=\{x\in\mathbb{R}^N:\ x\cdot e < c^*_e \text{ for every } e\in\mathbb{S}^{N-1}\},
\end{equation}
which is equivalent to
\[\mathcal{S}=\{re:\ e\in \mathbb{S}^{N-1},\, 0\leqslant r<s(e)\} \hbox{ with $c_e^*>0$ and $s(e)$ given by \eqref{FG}},\]
see \cite{R2017}. Recent works by Hamel and Rossi \cite{HR,HRJEMS} extended the Freidlin-G\"artner formula and the spreading set  to unbounded initial supports for \eqref{local-equation}.

As we mentioned, the framework above relies on the positivity of   all minimal wave speeds $c_e^*$,  which is equivalent to  the condition that the spreading set $\mathcal S$ contains the origin. This fact leads to  global invasion, that is, $u(t,x)\rightarrow 1$ locally uniformly in $\mathbb{R}^N$ as $t\rightarrow +\infty$. However, for \eqref{equation},
 the asymmetry of the kernel $\mathcal J$  breaks the rotational invariance of the equation.  The spreading set $\mathcal S$ may \textbf{not contain the origin}, which we call  \textbf{biased propagation}. In such a case, the origin-centered quantity $s(e)$ introduced
 through the Freidlin-G\"artner variational formula may be nonpositive or even equal to $-\infty$, and the section
$\mathcal S\cap\mathbb Re$ may even be empty.

In one space dimension, such a bias can essentially be described by two propagation directions\cite{XLR2021,DFN}, whereas in dimensions $N\geqslant2$ the invaded region has a genuinely multidimensional geometry and cannot in general be encoded by a pair of scalar speeds. Our primary objective is to establish a unified geometric description of spreading that encompasses  both classical origin-centered propagation and biased propagation induced by asymmetric dispersal.   We focus on the following fundamental problem: for the   nonlocal equation \eqref{equation} with  unbounded initial support \eqref{U} which  may exhibit biased propagation,   can one establish the existence and geometry  of the  spreading set, and ultimately derive  a generalized Freidlin-G\"artner-type formula?

 To describe the following results and for later convenience,  we recall some definitions in $\mathbb R^N$ and borrow some notation from \cite{HR}.
\subsection*{Notations} Let ``$\vert\cdot \vert$'' denote  the Euclidean norm, and
$
B_r(x):=\left\{y \in \mathbb{R}^N:|y-x|<r\right\}
$
 be the open Euclidean ball with  center $x \in \mathbb{R}^N$ and radius $r>0$ and $B_r:=B_r(0)$ for short. The distance from  $x \in \mathbb{R}^N$ to  $A \subset \mathbb{R}^N$ is denoted by  $\operatorname{dist}(x, A):=\inf \{|y-x|: y \in A\}$ with  the convention $\operatorname{dist}(x, \varnothing)=+\infty$. For two nonempty sets $A, B\subset \mathbb R^N$, we define their  Hausdorff distance  by
 \[d_{\mathcal H}(A,B)=\max\left\{\sup_{x\in A }\text{dist}(x,B),\sup_{y\in B }\text{dist}(y,A) \right\}\] with the convention $d_{\mathcal H}(\varnothing,\varnothing)=0$.
   The Minkowski sum of sets $A$ and $B$ is defined as
 $A+B:=\left\{x+y:x\in A, y\in B\right\}$.

To characterize spreading speeds, define the sets of bounded and unbounded directions as in \cite{HR}:
$$
\mathcal{B}(U):=\left\{e \in \mathbb{S}^{N-1}: \liminf _{\tau \rightarrow+\infty} \frac{\operatorname{dist}(\tau e, U)}{\tau}>0\right\}
$$
and
$$
\mathcal{U}(U):=\left\{e \in \mathbb{S}^{N-1}: \lim _{\tau \rightarrow+\infty} \frac{\operatorname{dist}(\tau e, U)}{\tau}=0\right\}.
$$
These sets characterize the asymptotic directions of $U$.
 We also define the notion of the  positive interior $U_\delta$($\delta>0$) of the set $U$ as
$U_\delta:=\{x \in U: \operatorname{dist}(x, \mathbb R^N\setminus U) \geqslant \delta\}$. For $\lambda\in(0,1)$ and $t>0$,  the upper level set of the solution to equation \eqref{equation} is defined by
  \begin{equation}\label{Flambda}
F_\lambda(t):=\left\{x \in \mathbb{R}^N: u(t, x)>\lambda\right\}
.
\end{equation}
 We set $\mathbb R^+:=[0,+\infty)$, $\mathbb R^+\mathcal U(U):=\{r\xi:r\geqslant 0,\xi\in\mathcal U(U)\}$ and for $x\in\mathbb R^N\backslash\{0\}$, $\hat x:=\frac{x}{|x|}.$

\subsection*{Outline of the paper}
The paper is organized as follows. Section~\ref{mainresults} states the main results, including  convergence to a linearly expanding set for compactly supported data,  and the variational characterization of directional spreading speeds and the geometric description of the spreading set for unbounded initial supports.  It also  provides illustrative examples, discussions, and additional remarks. Numerical simulations (see Figure~\ref{figure-1}) visually demonstrate these biased propagation scenarios under different kernel centers. Section~\ref{preliminaries} collects  the necessary preliminaries and auxiliary facts. Section~\ref{proofs} is devoted to the proofs of the main theorems, which rely on the properties of the intrinsic spreading set and the construction of anisotropic supersolutions.
\section{Main Results}\label{mainresults}

The existence of traveling fronts for nonlocal monostable equations with
possibly asymmetric dispersal has been extensively studied; see, for instance,
\cite{LPL2005,SLW2011,CDM2008,XLR2021,FKT2019,DFN,CD2007}.
Recall that an entire solution $u(t,x)$ of \eqref{equation} is called a
traveling front connecting $1$ to $0$ in the direction
$e\in\mathbb S^{N-1}$ with speed $c\in\mathbb R$ if
$u(t,x)=\phi_e(x\cdot e-ct)$,
where $\phi_e:\mathbb R\to(0,1)$ satisfies $
\phi_e(-\infty)=1$ and $\phi_e(+\infty)=0$.   Define the projected kernel in the direction $e$ by
\begin{equation}\label{projected-kernel}
J_e(r):=
\int_{e^\perp}
\mathcal J(re+z)\,d\mathcal H^{N-1}(z),\  e^\perp:=\{z\in\mathbb R^N:z\cdot e=0\},
\end{equation}
 for  almost every  $r\in\mathbb R$  and  $\mathcal H^{N-1}$ denotes the $(N-1)$-dimensional
Hausdorff measure on $e^\perp$. By Fubini's theorem and assumption (J),
\[
 J_e\in L^1(\mathbb R),\
J_e\geqslant0 \text{ a.e.},\
\int_{\mathbb R}J_e(r)\,dr=1,\
\int_{\mathbb R}J_e(r)e^{\mu|r|}\,dr<+\infty.
\]
The local nondegeneracy in (J) also implies that there exist $\delta_0,\rho_0>0$, independent of $e$, such that
$J_e(r)\geqslant\rho_0$
for a.e. $r\in(-\delta_0,\delta_0)$.
Moreover, if $x\cdot e=s$, then
$$
\int_{\mathbb R^N}
\mathcal J(x-y)\phi_e(y\cdot e-ct)\,dy
=
\int_{\mathbb R}
J_e(s-r)\phi_e(r-ct)\,dr.
$$
Hence the traveling-wave profile satisfies
\begin{equation}\label{traveling-wave}
\begin{cases}J_e*\phi_e-\phi_e+c \phi_e^{\prime}+f(\phi_e)=0,  \\ \phi_e(-\infty)=1>\phi_e(z)>\phi_e(+\infty)=0,\  z \in \mathbb{R},\end{cases}
\end{equation}
where $J_e*\phi_e(z)=\int_{\mathbb{R}}J_e(z-s)\phi_e(s)ds$.
The following standard result summarizes the traveling-wave
properties needed below; see
\cite{Y2009,FKT2019} and the references therein.
\begin{proposition}\label{Prop-1}
Assume {\rm (J)} and {\rm (F)}. Then, for every
$e\in\mathbb S^{N-1}$, there exists a minimal traveling wave
speed $c_e^*\in\mathbb R$, given by
\begin{equation}\label{c^*}
c_e^*=\inf_{\lambda>0}\frac{\displaystyle\int_{\mathbb R^N}\mathcal J(x)e^{\lambda x\cdot e}\,dx
-1+f'(0)}{\lambda}.
\end{equation}
More precisely, a nonincreasing traveling front connecting
$1$ to $0$ in the direction $e$ exists for every
$c\geqslant c_e^*$, whereas no such front exists for
$c<c_e^*$.
Furthermore, if $c\neq0$, then the corresponding profile
belongs to $C^1(\mathbb R)$.
\end{proposition}

We emphasize that the continuity of the projected kernel $J_e$ is not required for the regularity assertion in
Proposition~\ref{Prop-1}. We now turn to the geometry of the intrinsic spreading set of \eqref{equation} determined by the minimal wave speeds. The exponential-moment assumption in \rm (J) implies that the mean dispersal drift
$\mathbf p:=\int_{\mathbb R^N}x\mathcal J(x)dx$ is well defined.
\begin{proposition}\label{NewS}
Assume {\rm (J)} and {\rm (F)}. Define the spreading set  and  the origin-centered speed of equation \eqref{equation}:
\begin{equation}\label{se}\mathcal S
:=\left\{
x\in\mathbb R^N:
x\cdot e<c_e^*
\text{ for every }e\in\mathbb S^{N-1}
\right\},\ s(e):=\inf_{\xi\cdot e>0}
\frac{c_\xi^*}{\xi\cdot e}
\in[-\infty,+\infty).\end{equation}
\begin{itemize}
\item[(i)]
The set $\mathcal S$ is a nonempty bounded open convex set and  $\mathbf p\in\operatorname{int}\mathcal S$.
In particular, $\mathcal S$ is star-shaped with respect to $\mathbf p$.
\item[ (ii)]
For every $e\in\mathbb S^{N-1}$,
\begin{equation}\label{origin-ray-section}
\mathcal S\cap\mathbb Re=\begin{cases}
\{re:-s(-e)<r<s(e)\},
&\text{if }c_\xi^*>0\text{ for every }\xi\cdot e=0,\\
\varnothing,& \text{otherwise}.
\end{cases}
\end{equation}
Thus the origin-centered  description may degenerate when $0\notin\mathcal S$.
\item[\rm (iii)]
Define the drift-centered  speed by $
\sigma(e) :=\sup\{r\geqslant0:\mathbf p+re\in\mathcal S\}$.
Then
\begin{equation}\label{sigma-FG}
\sigma(e)=\inf_{\xi\cdot e>0}
\frac{c_\xi^*-\mathbf p\cdot\xi}{\xi\cdot e}\in(0,\infty).
\end{equation}
Furthermore, $\mathcal S=\left\{
\mathbf p+re:e\in\mathbb S^{N-1},\ 0\leqslant r<\sigma(e)\right\}.$
\end{itemize}
\end{proposition}
The quantities $s(e)$ and $\sigma(e)$ describe the radial geometry of $\mathcal S$ with respect to the origin and the drift $\mathbf p$, respectively. In the biased regime, $s(e)$ may be nonpositive or even $-\infty$, whereas
$\sigma(e)\in(0,\infty)$ for every $e\in\mathbb S^{N-1}$ since $\mathbf p\in\operatorname{int}\mathcal S$. Moreover,
$\sigma(e)$ is precisely the distance from $\mathbf p$ to
$\partial\mathcal S$ in the direction $e$.

It is worth noting that if $f'(0)$ is greater than or equal to $1$, the spreading  set always contains $0$ and all minimal wave speeds are positive. As an immediate consequence of Proposition \ref{NewS},
$c_e^*+c_{-e}^*>0$ for every $e\in\mathbb S^{N-1}$.

\subsection{Compactly Supported Initial Data}
 For compactly supported initial data, we  establish the convergence of the solution  to the linearly expanding set $t\mathcal S$. In particular, this convergence holds even when $\mathcal{S}$ does not contain the origin, which corresponds to the case of biased propagation.

\begin{theorem}\label{invasionS}
Assume that (J) and (F) hold,   $U$ is a compact set with $U_\delta\not = \varnothing$ for some $\delta>0$. Let $\mathcal S'$ and $\mathcal S''$ be any  nonempty open sets  such that $\overline {\mathcal S''}\subset \mathcal S\subset\overline{ \mathcal S}\subset \mathcal S'$, where $\mathcal S$ is given by \eqref{mathcals}. Then $\mathcal S$ is the intrinsic spreading set of \eqref{equation},  and  the solution $u(t,x)$  satisfies:
\begin{equation}\label{S}
 \limsup_{t\to+\infty}\left[ \sup_{x\not\in t\mathcal S'} u(t,x)\right]= 0,\
  \liminf_{t\to+\infty}\left[ \inf_{x\in t\mathcal S''} u(t,x)\right]= 1.
\end{equation}
\end{theorem}

While results of this type are well known for local diffusion\cite{W1982,W2002,R2017,DuLiShen2022,DuLiXin2026,LZ2007}, Theorem \ref{invasionS} is novel in the context of biased propagation driven by asymmetric nonlocal diffusion, and is established via a geometric approach.  The following  scenarios in Corollary \ref{corollary-1} vividly illustrate the so-called biased propagation, which depends on whether the origin belongs to $\mathcal S$.

\begin{corollary}\label{corollary-1}
Let $u(t,x)$ be the solution to \eqref{equation} in Theorem \ref{invasionS}. We always have\[
\lim_{t\to+\infty}
u(t,t(\mathbf p+re))=1
\text{ for every } r\in(-\sigma(-e),\sigma(e)) \text{ and }e\in\mathbb S^{N-1}.
\]
On the other hand, it follows that
\begin{itemize}
\item[1.]
if $0\in\mathcal S$, then $\lim\limits_{t\to+\infty}u(t,x)=1$ locally uniformly in  $\mathbb R^N$;
\item[2.] if $0\notin \overline{\mathcal S}$, then $\lim\limits_{t\to+\infty}u(t,x)= 0$ locally uniformly in $\mathbb R^N$.
 \end{itemize}
\end{corollary}

In this paper, condition (J) ensures several key analytical features.
Under these assumptions, the minimal wave speed $c^*_e$ remains bounded for all $e\in\mathbb S^{N-1}$, and the intrinsic spreading set $\mathcal S$ is bounded with nonempty interior.  Moreover, the solution exhibits invasion in the sense that $\lim\limits_{t\to+\infty}u(t,tx)=1$ locally uniformly for all $x\in\mathcal S$. The exponential-moment condition in (J) is essential for the
finite-speed propagation framework considered here.
Heavy-tailed kernels (e.g., those with algebraic decay) preclude the existence of traveling waves  in the standard
monostable finite-speed setting considered here altogether~\cite{G2011,Y2009,CR2013}. In that regime, the  description of speeds and sets used
throughout this paper may break down, and accelerating propagation can occur. We leave the analysis of accelerated
propagation for asymmetric heavy-tailed kernels to future work.

\subsection{Unbounded Initial Support} Our main contribution concerns solutions with unbounded initial support $U$. In the sequel, we impose a mild geometric regularity condition on the spreading set~$\mathcal S$,  which is satisfied for important classes of sufficiently regular kernels, including the Gaussian kernels considered below.
\begin{hypothesis}\label{mainhypothesis} The intrinsic spreading set $\mathcal S$ of equation~\eqref{equation} given by \eqref{mathcals} has a   $C^1$ boundary  $\partial \mathcal S$.
\end{hypothesis}
This  $C^1$ regularity holds for sufficiently smooth kernels such as Gaussian kernels \eqref{Gaussian}\cite{S2014}. It implies that   for any  $e\in\mathbb S ^{N-1}$, the hyperplane $\{x\in\mathbb R^N: x\cdot e=c^*_e\}$ is a supporting hyperplane of $\mathcal S$, i.e.,
$$\partial \mathcal S \cap \{x\in\mathbb R^N: x\cdot e=c^*_e\}\neq \varnothing,$$
see \cite[Proposition 5.6]{GR2025}.
 Furthermore, under Hypothesis \ref{mainhypothesis}, for any $z\in\partial \mathcal S$,  the exterior  unit normal $\nu$ at $z$ is unique  and
  \[\sup_{x \in \overline{\mathcal{O}}} x \cdot \eta = c_{\eta}^{*} - \mathbf{p} \cdot \eta,\ \mathcal O:=\mathcal S-\mathbf p.
  \]If  a line passing through  the origin intersects $\overline{\mathcal S}$ at $z\in\partial \mathcal S\backslash \{0\}$ (i.e., $z=r\hat z$ for some $r$), then
\begin{equation}\label{C1-2}z\cdot \nu =c^*_{\nu}\in\mathbb R,\ z=\begin{cases} s(\hat z)\hat z,\ \text{ if } z\cdot \nu>0,\\
-s(-\hat z)\hat z,\ \text{ if } z\cdot \nu\leqslant0.
\end{cases}.
 \end{equation}Without $C^1$-regularity, the boundary may have corners  or flat pieces and the supporting normal may fail to be unique; this uniqueness is used below to single out the minimizing direction in \eqref{se} and to build anisotropic supersolutions (see the proofs of the theorems below).
 Hypothesis \ref{mainhypothesis} and \eqref{C1-2} do not require the origin to lie inside $\mathcal S$.

The next theorem ensures the existence of such a spreading set under a geometric condition.
\begin{theorem}[\textbf{Existence and representation of  the spreading set}]\label{th2} Assume that (J), (F) and Hypothesis~\ref{mainhypothesis} hold. Let  $u(t,x)$  solve \eqref{equation} with  \eqref{U}, where  $U \subset \mathbb{R}^N$ satisfies $U_\delta \neq \varnothing$ for some $\delta>0$ and
\begin{equation}\label{SN-1}
\mathcal{B}(U) \cup \mathcal{U}\left(U_\delta\right)=\mathbb{S}^{N-1}.
\end{equation}   Then there exists a spreading set $\mathcal{W}$ for the solution to \eqref{equation}  defined by
\begin{equation}\label{th2-1}
\mathcal{W}:=\mathbb{R}^{+}\mathcal{U}(U)+\mathcal S.
\end{equation}
This set can be represented geometrically as:
\begin{equation}\label{th2-2}
\mathcal{W}=\left\{\mathbf p+ re:  e \in \mathbb{S}^{N-1},\ -w(-e) < r<w(e)\right\},
\end{equation} with $w(e)\in[\sigma(e),+\infty]$  given by
 \begin{equation}\label{th1-2}
w(e)=\sup_{
    \xi \in \mathcal{U}(U)
}  \inf_{\eta \cdot e > 0,
\eta \cdot \xi\leqslant 0} \frac{c_\eta^* -\mathbf p \cdot \eta}{\eta \cdot e},
\end{equation}
under the standard conventions $\inf \varnothing = +\infty$ and $w(e) = \sigma(e)$ if $\mathcal{U}(U) = \varnothing$,
with the convention that $\mathbb{R}^{+}\varnothing+\mathcal S=\mathcal S.$ Moreover, for any  $\lambda\in(0,1)$ and   any compact set $\mathcal K\subset\mathbb{R}^N$ satisfying that  $\overline{\mathcal K\cap\mathcal{W}}=\mathcal K\cap\overline{\mathcal{W}}$, we have
 \begin{equation}\label{th2-3}
 d_{\mathcal{H}}\left(\mathcal K\cap \frac{1}{t}F_\lambda(t), \mathcal K\cap\mathcal{W}\right) \to 0 \text{ as } t\to +\infty.
 \end{equation}
\end{theorem}
Although $\mathcal W$ is in general nonconvex
(see Subsection~\ref{examples}), it is star-shaped with respect to $\mathbf p$. Writing $\mathcal O=\mathcal S-\mathbf p$, we have $\mathcal W-\mathbf p =\mathbb R^+\mathcal U(U)+\mathcal O$.
Since $\mathbb R^+\mathcal U(U)$ is a cone and $\mathcal O$ is convex with $0\in\mathcal O$,
$\theta(\mathcal W-\mathbf p)\subset\mathcal W-\mathbf p$
for every $\theta\in[0,1]$.
Hence $w(e)$ describes the radial extent of $\mathcal W$ in the frame centered at $\mathbf p$. In general, it is not an
origin-centered spreading speed, since $\mathcal W$ need not be star-shaped with respect to the origin. When $\mathbf p=0$, the two frames coincide and $w(e)$ reduces to the usual radial spreading speed.

If $e\in\mathcal U(U)$, then $w(e)=+\infty$. Since
$\mathbf p\in\mathcal S$ and $\mathbb R^+e\subset\mathbb R^+\mathcal U(U)$,
\[
\mathbf p+\mathbb R^+e\subset\mathbb R^+\mathcal U(U)+\mathcal S=\mathcal W.
\]
Thus $\mathcal W$ is unbounded in the direction $e$ from the center $\mathbf p$. This also agrees with \eqref{th1-2}, because taking $\xi=e$ makes the admissible set
$\{\eta\in\mathbb S^{N-1}:
\eta\cdot e>0,\ \eta\cdot\xi\leqslant0\}$
empty, so the corresponding infimum is $+\infty$.

Theorem~\ref{th2} also yields the following consequence
concerning local invasion and extinction.
 \begin{corollary}\label{corollary-2}
Assume that the hypotheses of Theorem~\ref{th2} hold.
If $0\in\mathcal W$, then
$u(t,x)\to1$ as $t\to+\infty$
locally uniformly in $\mathbb R^N$. Moreover, the condition $0\in\mathcal W$ is equivalent to either $0\in\mathcal S$, or there exists $x\in\mathcal S$ such that
$-\widehat{x}\in\mathcal U(U)$.
 If $0\notin\overline{\mathcal W}$, then
$u(t,x)\to0$
as $t\to+\infty$ locally uniformly in $\mathbb R^N$. \end{corollary}
The truncation by the compact set~$\mathcal K$ is necessary for \eqref{th2-3} to hold; in general,
 \[d_{\mathcal{H}}\left( \frac{1}{t}F_\lambda(t), \mathcal{W}\right) \not \to 0  \text{ as } t\to +\infty.\]
 However, under another mild geometric assumption on~$U$, one can recover a global approximation of the level sets by a suitable family of translated sets.

\begin{theorem}\label{th3} Assume that (J) and (F) hold. Let  $u(t,x)$  solve \eqref{equation} with  \eqref{U}. Assume that  $U_\delta\neq\varnothing$ for some $\delta>0$ and
\begin{equation}\label{finite-hausdorff}
d_{\mathcal{H}}(U,U_\delta)<+\infty.
\end{equation}
Then for any $\lambda\in(0,1)$, we have
 \begin{equation}\label{th3-1}
 d_{\mathcal{H}}\left(F_\lambda(t), U+t\mathcal S \right)= o(t) \text{ as  } t\to+\infty.
 \end{equation}
 \end{theorem}

 Property~\eqref{th3-1} means that, for large time, $F_\lambda(t)$ behaves like the support set $U$ thickened by the expanding region $t\mathcal S$ under~\eqref{finite-hausdorff}.
A sufficient condition for~\eqref{finite-hausdorff} to hold is that $U$ satisfies the uniform interior sphere condition with radius $\delta$.
Moreover, Theorem~\ref{th3} also applies to any nonempty set $U$ that is uniformly $C^{1,1}$. Under condition~\eqref{SN-1}, we have
\begin{equation}\label{udelta=u}\mathcal{U}(U) \supset \mathcal{U}\left(U_\delta\right)=\mathbb S^{N-1} \backslash \mathcal{B}(U) \supset \mathcal{U}(U)
\end{equation}
On the other hand, condition~\eqref{finite-hausdorff} implies
\[\liminf_{\tau\to+\infty}\frac{\text{dist}(\tau e,U)}{\tau}=0\Longleftrightarrow\liminf_{\tau\to+\infty}\frac{\text{dist}(\tau e,U_\delta)}{\tau}=0
\]
and consequently $\mathcal{U}(U_\delta)=\mathcal{U}(U)$. Hence, both assumptions~\eqref{SN-1} and~\eqref{finite-hausdorff} essentially guarantee that small or thin parts of $U$ (such as isolated points or narrow seams) do not affect the unbounded directions of propagation.

Although the spreading speed defined as \eqref{sspeed} may not exist for \eqref{equation}, we can still  regard $w(e)$ as a generalized  speed by Theorem~\ref{th2}, see also \cite{GHR2026} for the formula of $w(e)$ in the spatially periodic case.

\begin{theorem}[\textbf{Variational characterization of  speeds}]\label{th1} Assume that  the hypotheses of Theorem \ref{th2} hold. Then, for every $e\in\mathbb S^{N-1}$,  the solution satisfies
\begin{equation}\label{th1-1}
\left\{\begin{aligned}
\lim \limits_{t \rightarrow+\infty} u(t, t(\mathbf p+re))=1 \text{ for any }r\in(-w(-e),w(e)),\\
\lim \limits_{t \rightarrow+\infty} u(t, t(\mathbf p+re))=0 \text{ for any } r\not\in[-w(-e),w(e)].
\end{aligned}
\right.
\end{equation}
where $w(e)$ is given by \eqref{th1-2}.
\end{theorem}

From~\eqref{th1-2} we immediately obtain $w(e)\geqslant \sigma(e)$  when $\mathcal U(U)\neq\varnothing$.
In particular, if $\mathcal{U}(U)=\varnothing$, then $w(e)=\sigma(e)$ for all $e$, which recovers  Theorem \ref{invasionS} in the  compactly supported case. In the isotropic case---for instance, for~\eqref{local-equation} or  a radially symmetric kernel $\mathcal J$---the minimal wave speed $c^*_e\equiv c^*>0$ is independent of direction, $\mathbf p=0$ and $\mathcal S=B_{c^*}$.
Then  $we -r\xi \in B_{c^*}$ means that  the distance from $we$ to the ray $\mathbb{R}^+ \xi$ is less than $c^*$. The perpendicular distance from $we$ to the line $\mathbb{R}\xi$ is $d =w\sqrt{1-(\xi\cdot e)^2}$. Requiring $d<c^*$ and $w<w(e)$ directly yield:
\[
w(e)=\max\{c^*,\ \sup_{\xi\in\mathcal{U}(U),\,\xi\cdot e\geqslant0}
\frac{c^*}{\sqrt{1-(\xi\cdot e)^2}}\},
\]
which  agrees with the isotropic results of Hamel and Rossi~\cite{HR}.  To account for biased propagation, the convergence in \eqref{th1-1} is stated for the asymmetric interval $(-w(-e),w(e))$. The geometric condition $U_{\delta}\neq\varnothing$ in Theorem \ref{th2} is not merely technical. For KPP-type nonlinearities, the so-called \emph{hair-trigger effect} ensures that propagation can be initiated if the initial support has a positive measure\cite{W1982,Z2006,A2017Poincare,ADK2023,FT2018,AW1978}.

The geometric assumptions~\eqref{SN-1} and~\eqref{finite-hausdorff} play complementary but distinct roles. Both remain invariant under rigid transformations of~$U$, yet neither implies the other.
Condition \eqref{SN-1}, which ensures that every direction is either bounded or unbounded for $U_\delta$, is substantially stronger than the mere requirement that $U$ is unbounded.
Replacing~$\mathcal U(U_\delta)$ with~$\mathcal U(U)$ in this condition invalidates key convergence results~\cite[Remark 6.4]{HR}.
In contrast,~\eqref{finite-hausdorff} is satisfied under uniform interior sphere conditions.
Theorems~\ref{th2}--\ref{th1} characterize the
spreading behavior and the asymptotic shape of the upper level sets, either in terms of $\mathcal W$ or of $U+t\mathcal S$.
The global convergence~\eqref{convergence}, however, may fail in general; see~\cite{HR}.

The conclusions of Theorems~\ref{th2}-\ref{th1} still hold for  solutions to~\eqref{equation} with  more general initial conditions. For example,  suppose  that  for some $\theta\in(0,1]$ and $\delta>0$, the following hold:
$$
\left\{u_0 \geqslant \theta\right\}_\delta \neq \varnothing, \quad \mathcal{B}\left(\left\{u_0 \geqslant \theta\right\}\right) \cup \mathcal{U}\left(\left\{u_0 \geqslant \theta\right\}_\delta\right)=\mathbb{S}^{N-1}
$$
and
$$
d_{\mathcal{H}}\left(\operatorname{supp} u_0,\left\{u_0 \geqslant \theta\right\}\right)<+\infty.
$$
Then the conclusions of Theorems~\ref{th2}--\ref{th1} hold, with $\mathcal{U}(U)$ replaced by $\mathcal{U}\left(\left\{u_0 \geqslant \theta\right\}\right)$ in the definitions  of $w(e)$, following the same arguments as in~\cite[Remark 5.2]{HR}.

In this work, we only consider the  KPP-type nonlinearity $f$ and Proposition~\ref{Prop-1} always holds for any $e\in\mathbb S^{N-1}$. If $f$ is of ignition type, i.e.,
\begin{equation}\label{ignition}
 \text{there exists }\alpha \in(0,1)\text{ such that } f=0 \text { on }[0, \alpha] \text { and } f>0 \text { in }(\alpha, 1),
\end{equation}
or if $f$ is of bistable type, i.e.,
\begin{equation}\label{bistable}
 \text{there exists }\alpha \in(0,1)\text{ such that } f<0 \text { on }(0, \alpha) \text { and } f>0 \text { on }(\alpha, 1) \text{ with } \int_0^1 f>0,
\end{equation} then~\eqref{traveling-wave} admits traveling fronts connecting $1$ to $0$ with a unique constant speed $c^*_e$ (which may be  non-positive), see~\cite{C2007,BFRW1997,S1980,CD2007}.   However, unlike in the KPP case, the nonemptiness of the spreading set is not guaranteed by a direct verification, and would require a separate analysis depending on the specific nonlinearity and kernel.

 For the  above types of $f$, Theorem \ref{invasionS} does not hold for arbitrary  nonempty $u_0$. We additionally require  the initial support $U$ to be ``sufficiently  large'' to ensure  invasion. This   is related to sharp-threshold initial conditions  for  equation~\eqref{equation} with bistable and ignition-type nonlinearities~\cite{Z2006,DM2010,ADK2023,DLL2025,DP2015}. Therefore, we only consider the KPP case in this work. The extension of the present geometric framework to ignition or bistable nonlinearities, which involves both establishing the nonemptiness and geometric properties of $\mathcal S$ and adapting the propagation dynamics for large initial data, remains an interesting direction for future study.

Our theoretical framework is not restricted to the specific nonlocal operator considered. Although the analysis focuses on nonlocal operators $\mathcal J\star u-u$, the geometric-variational approach for constructing the spreading set $\mathcal W$ and the directional speed $w(e)$ relies essentially on three properties: (i) existence of well-defined wave speeds $c_e^*$;  (ii) convexity and regularity of the intrinsic spreading set $\mathcal S$; and (iii) validity of comparison principles.  Consequently, this framework applies naturally to other diffusion models that support traveling fronts and satisfy (i)-(iii),
 such as equations involving spatio-temporally nonlocal operators.

 \subsection{Examples}\label{examples}The geometry of the spreading set $\mathcal S$ plays a crucial role in determining the asymptotic behavior of solutions. In the simplest case, when $\mathcal J$ is radially symmetric, equation~\eqref{traveling-wave} admits a solution $(\phi, c^*)$ for every direction $e\in\mathbb S^{N-1}$. The spreading set $\mathcal S$ of any invading solution to~\eqref{equation} with compactly supported initial conditions exists and equals
\[\mathcal S=B_{c^*},\]
where $c^*$ is the minimal wave speed. We now consider more general kernels that break radial symmetry. For example, consider the multivariate Gaussian kernel with $f'(0)=1$
\begin{equation}\label{Gaussian}
\mathcal J(x)=\frac{1}{(2\pi)^{N/2}|\Sigma|^{1/2}}\exp(-\frac{1}{2} (x-\mu)^T\Sigma^{-1}(x-\mu)), x\in\mathbb R^N,
\end{equation}
 with mean $\mu$ and covariance matrix $\Sigma$. Set $\sigma_e^2=e^T\Sigma e$. A direct computation yields
  \[\begin{aligned}\Lambda_e(\lambda) =\exp\left(\lambda(\mu\cdot e)+\frac{1}{2}\lambda^2\sigma_e^2\right),\
  c_e^*=\inf_{\lambda>0}G_e(\lambda)=\inf_{\lambda>0}\frac{\exp\left(\lambda(\mu\cdot e)+\frac{1}{2}\lambda^2\sigma_e^2\right)}{\lambda},\\
  G'(\lambda) = \frac{ (\lambda^2 \sigma_e^2+\lambda\mu \cdot e -1) e^{\lambda(\mu \cdot e) + \frac{1}{2} \lambda^2 \sigma_e^2} }{\lambda^2},\ \lambda_e^* = \frac{-(\mu \cdot e) + \sqrt{(\mu \cdot e)^2 + 4\sigma_e^2}}{2\sigma_e^2}.
  \end{aligned}\]
 If $\mu=0$, the kernel is centered at the origin. Then $\lambda_e^* = 1/\sigma_e$, and the minimal wave speed simplifies to $c_e^* = \sigma_e\sqrt{\exp(1)}$. In this case, the spreading set $\mathcal{S}$ is a standard ellipsoid centered at the origin:
 \[\mathcal S = \{x\in\mathbb R^N : x\cdot e < \sqrt{\exp(1) \cdot e^T \Sigma e}, \forall e\in\mathbb{S}^{N-1} \} = \{x\in\mathbb R^N: \frac{1}{\exp(1)}x^T\Sigma^{-1}x < 1\}.\]

 Another important example satisfying Hypothesis \ref{mainhypothesis} is the multivariate Laplace kernel:
  \begin{equation}\label{laplace}\mathcal{J}(x) = \prod_{i=1}^{N} \frac{\lambda}{2} e^{-\lambda |x_i|}, \quad \lambda > 0.\end{equation}
  In this case, the spreading set $\mathcal{S}$ is a strictly convex set that can be viewed as a rounded octahedron centered at the origin.

For ignition or bistable nonlinearities, even when planar fronts exist, a particular solution may fail to invade if the initial datum lies below the corresponding threshold. Thus the existence of a nonempty intrinsic geometric candidate does not by itself guarantee invasion from arbitrary initial data.

The two geometric conditions \eqref{SN-1} and \eqref{finite-hausdorff} are independent in general. The following examples, which extend those in \cite{HR} to our nonlocal setting, illustrate typical scenarios where these conditions are met and clarify their geometric meaning.
Nevertheless, under additional structural assumptions, one may imply the other, as shown in~\cite[Proposition 5.1]{HR} and in the results below.

Even when~\eqref{SN-1} holds, the question of  existence of a spreading set is much more intricate when $U$ is unbounded.  Various conditions for the existence of a spreading set  for  equation~\eqref{local-equation} with unbounded $U$ have been given in \cite{HR}. We will also discuss the  validity of  the following convergences in Hausdorff:
\begin{equation}\label{convergence}
\lim\limits_{t\to+\infty}\frac{1}{t}F_\lambda(t)=\mathcal W=\lim\limits_{t\to+\infty}\frac{1}{t}U+\mathcal S,
\end{equation}
which are expected to hold but   actually fail in general.

We now exhibit some  examples where some conclusions of Theorems \eqref{th2}-\eqref{th1} may fail and \eqref{SN-1} or \eqref{finite-hausdorff} are not satisfied for our nonlocal equation. Up to rotation, let us  first consider the \textit{subgraphs} of functions in direction $x_N$, i.e.,
$$
U=\left\{x=(x',x_N) \in \mathbb{R}^{N-1}\times\mathbb R:x_N \leqslant \gamma(x')\right\} .
$$
with $\gamma\in L_{\text{loc}}^\infty(\mathbb R^{N-1})$ such that $\gamma(x')/|x'|\to\alpha\in[-\infty,+\infty]$. Then condition \eqref{SN-1} is fulfilled and the spreading set $\mathcal W$ exists. However, the shape of the spreading set  $\mathcal W$ given by \eqref{th2-2} is completely different according to the sign and boundedness of $\alpha$. If $\alpha=+\infty$ (for example, $\gamma(x')=|x'|^2$), then $\mathcal W=\mathbb R^N$; if $\alpha=-\infty$ (for example, $\gamma(x')=-|x'|^2$), then $\mathcal W=\left\{x\in\mathbb R^N:|x'|=0, x_N\leqslant0\right\}+\mathcal S$. When $0<\alpha<+\infty$, then $\mathcal W$ is a shifted  cone, that is $$\mathcal W=\mathcal S+\mathcal B,\ \mathcal B:=\left\{x=(x',x_N) \in \mathbb{R}^{N-1}\times\mathbb R:x_N \leqslant \alpha |x'|\right\},$$
hence $\mathcal W$ is non-convex and not $C^1$. If $\alpha<0$, the spreading set $\mathcal W$ is still given by the $\mathcal S$-neighborhood   of the same cone $\mathbb R^+\mathcal U(U)$, and it becomes ``rounded'' in its upper part, and moreover $\mathcal W$  is convex. Finally, if $\alpha=0$ (for example, $\gamma(x')=c^*$ or $\gamma(x')=\sqrt{|x'|}$), then $\mathcal W$ is  given by the half-space $\{x\in\mathbb R^N, x_N<c_{e_N}^*\}$ with $e_N=(0,\cdots,0,1)$. Note that  if $\gamma(x')$ is assumed to have uniformly bounded local oscillations,
$$\sup\limits_{x',y'\in\mathbb R^{N-1}, |x'-y'|\leqslant1}|\gamma(x')-\gamma(y')|<+\infty,$$
then the condition \eqref{finite-hausdorff} is fulfilled\cite{HR}. Here is  a two-dimensional scenario where the unbounded direction is fixed as $(-1,0)$, while all other directions remain bounded. We present four Gaussian kernel functions with distinct biased centers, see Figure \ref{figure-1}.

 \begin{figure}
\centering
\subfigure[Centered at the origin, symmetric propagation.]{
  \includegraphics[scale=0.13]{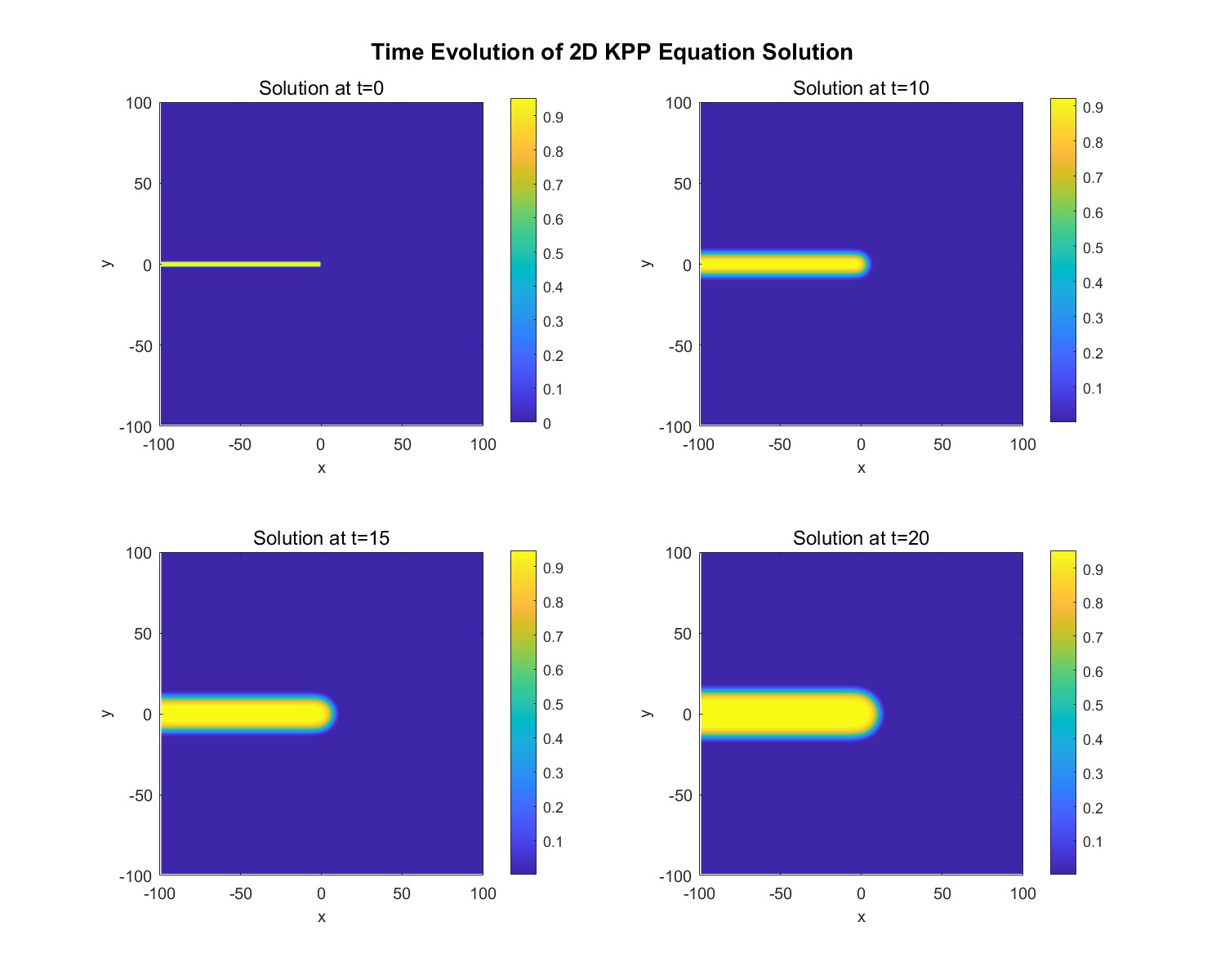}
}
\subfigure[Horizontally biased, rightward-biased propagation.]{
  \includegraphics[scale=0.13]{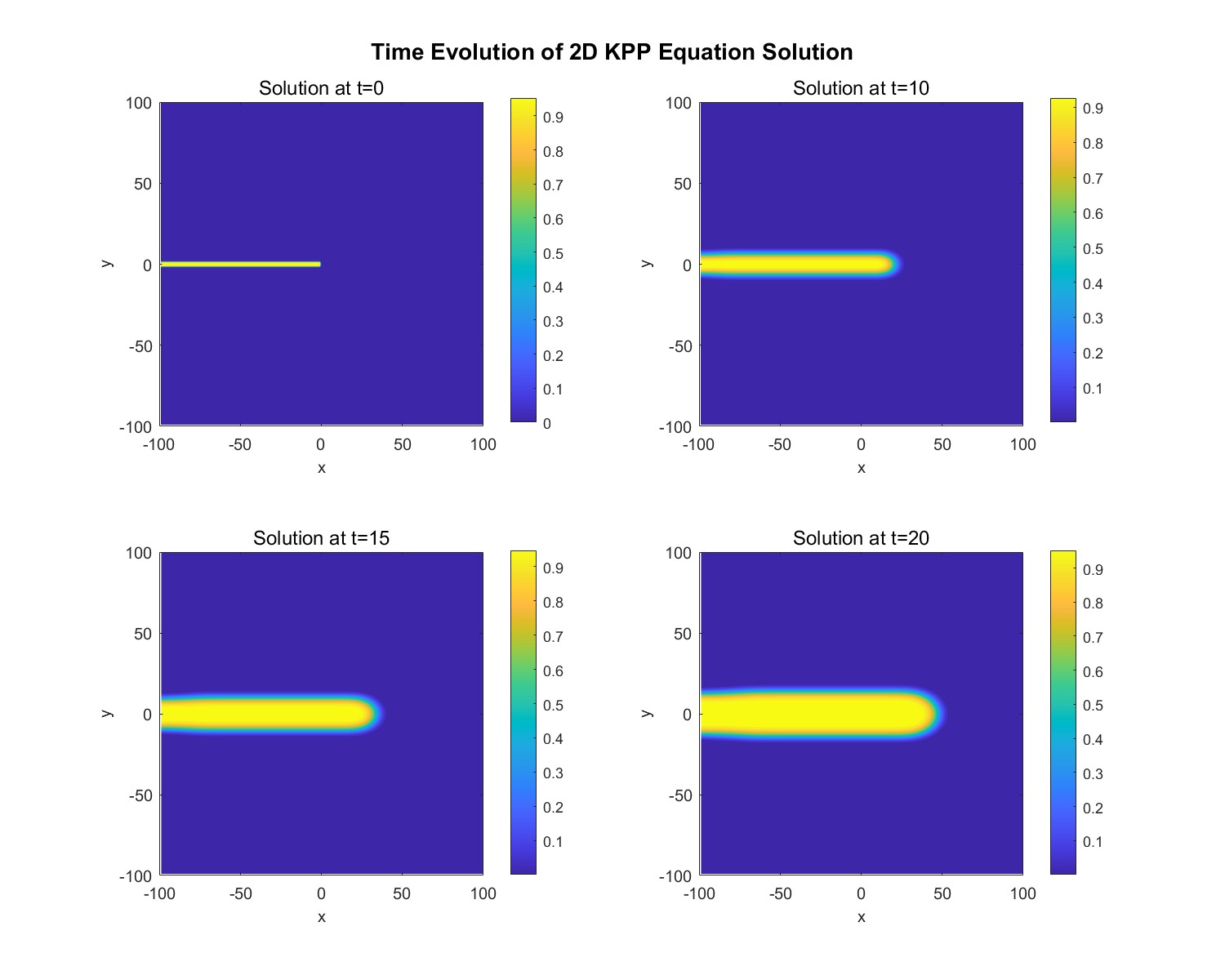}
}
\subfigure[Vertically  biased, upward-biased propagation.]{
  \includegraphics[scale=0.13]{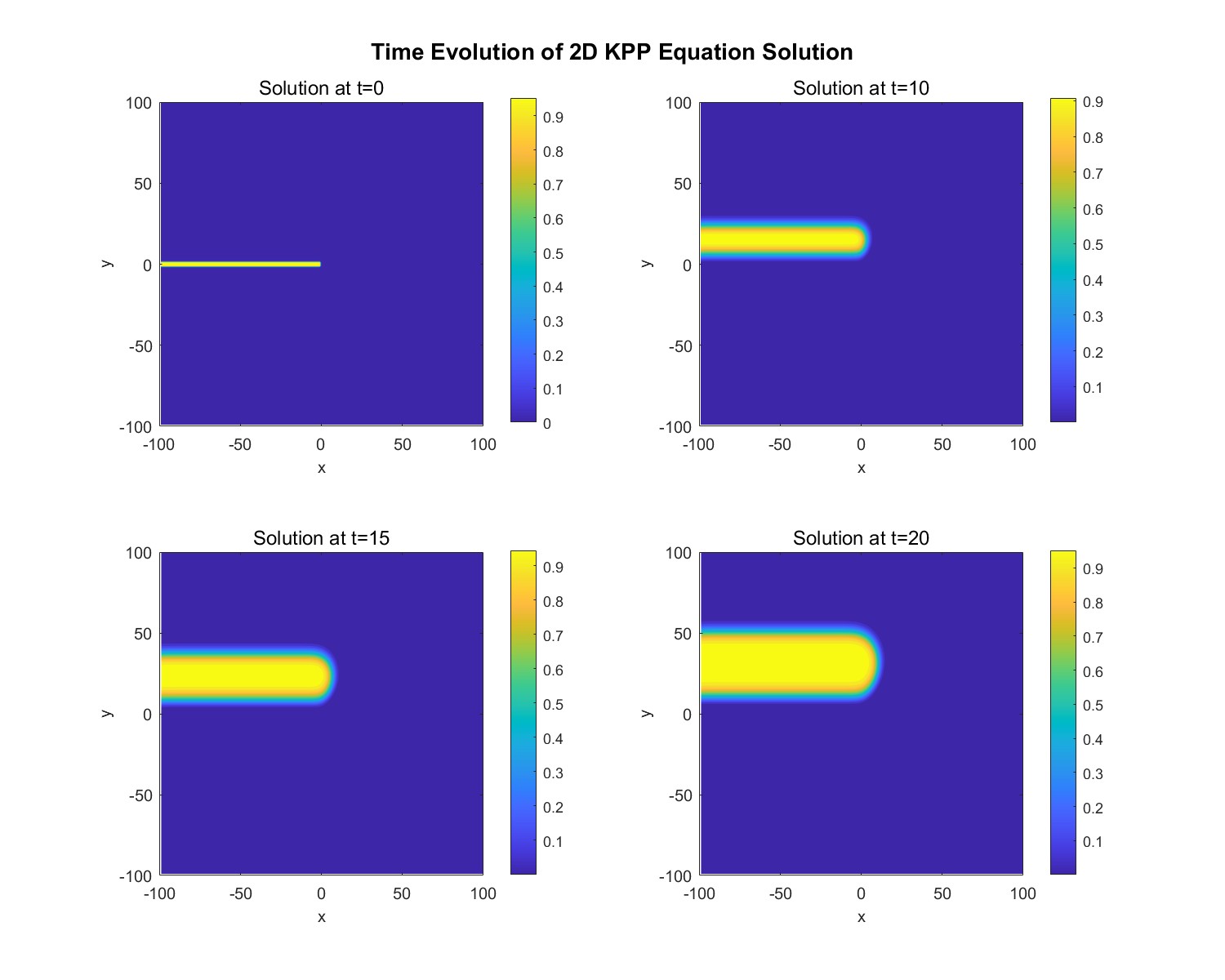}
}
\subfigure[Diagonally biased,  northeast-biased propagation.]{
  \includegraphics[scale=0.13]{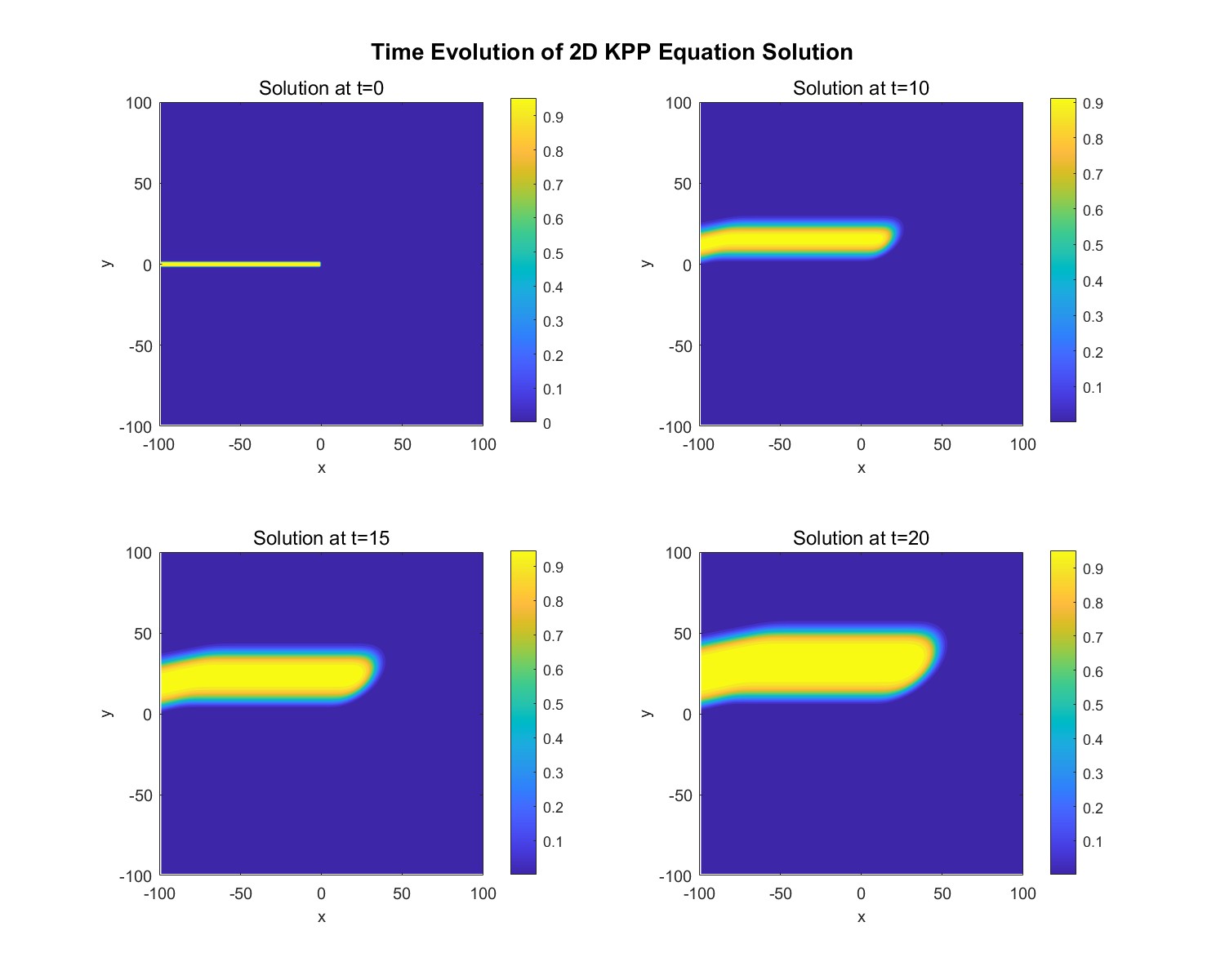}
}
\caption{Time evolutions of 2D KPP nonlocal equations  for Gaussian kernels with different centers. The unbounded direction is fixed as $(-1,0)$.}\label{figure-1}
\end{figure}

\section{Preliminaries}\label{preliminaries}

In addition to existence and uniqueness, we study the continuity of the maps  $c^*_e:\mathbb S^{N-1}\to \mathbb R$.  Although continuity of the wave speed with respect to direction is often implicitly assumed in periodic or Fisher-KPP settings~\cite{BF2012,W2002}, it has been rigorously established in several cases, including ignition-type reactions~\cite{AG2016}, strongly bistable reactions~\cite{G2018}, and more general cases~\cite{GHR2025}.

Define the minimal wave speed $c_e^*$ and the moment generating function $\Lambda_e(\lambda)$ by
\[\begin{aligned}\Lambda_e(\lambda):=\int_{\mathbb R^N}\mathcal J(y)e^{\lambda(y\cdot e)}dy=\int_{\mathbb R}J_e(z)e^{\lambda z}dz,\ \mathcal D_e:=\{\lambda>0:\Lambda_e(\lambda)<+\infty
\},\\
\lambda_e^+:=\sup\mathcal D_e\in(0,+\infty],\ c_e^*:=\inf_{\lambda\in\mathcal D_e}G_e(\lambda)=\inf_{\lambda\in\mathcal D_e}\frac{\Lambda_e(\lambda)-1+f'(0)}{\lambda},\
\end{aligned}\]
By the assumption (J), one has $\mathcal D_e\neq\varnothing$ and $\mathcal D_e$  is an interval of the form
\[(0,\lambda_e^+) \subset\mathcal D_e \subset  (0,\lambda_e^+].
\]
\begin{proposition}\label{Lambda}Assume that $\mathcal J$ satisfies (J) and $f'(0)>0$. Then:
\begin{itemize}
\item[1.] $G_e,\Lambda_e\in C^{\infty}((0,\lambda_e^+))$ and $\Lambda_e(\lambda)$ is strictly convex on $(0,\lambda_e^+)$.
\item[2.]$G_e(\lambda)$ is either strictly decreasing on $\mathcal D_e$, or admits a unique $\lambda_e^*\in\mathcal D_e$ such that $G_e(\lambda)$ is strictly decreasing on $(0,\lambda_e^*)$ and strictly increasing on $(\lambda_e^*,\lambda_e^+)$. Moreover, $\lim\limits_{\lambda\to0^+}G_e(\lambda)=+\infty.$
    \item[3.] The limit $\lim\limits_{\lambda\uparrow\lambda_e^+}\Lambda_e(\lambda)$ exists in $(0,+\infty]$ and $\lim\limits_{\lambda\to+\infty}\Lambda_e(\lambda)=+\infty$ if $\lambda_e^+=+\infty$. Moreover,
             \begin{itemize}
             \item[i.]if $\lim \limits_{\lambda \uparrow \lambda_e^{+}} \Lambda_e(\lambda)=+\infty$, then
$\lim \limits_{\lambda \uparrow \lambda_e^{+}} G_e(\lambda)=+\infty$,
 and there exists a unique $\lambda_e^* \in\mathcal D_e$ such that
$$
c_e^*=G_e\left(\lambda_e^*\right);
$$

\item[ii.] if $\lambda_e^{+}<+\infty$ and $\lim \limits_{\lambda \uparrow \lambda_e^{+}} \Lambda_e(\lambda)<+\infty$,
    then  $\lambda_e^+\in \mathcal D_e$, and either $
c_e^*=G_e\left(\lambda_e^*\right)$ for some unique $\lambda_e^* \in (0,\lambda_e^+)$, or
$c_e^*= G_e(\lambda_e^+)$.
\end{itemize}
\end{itemize}
\end{proposition}Proposition \ref{Lambda} gives a complete characterization of the possible behaviors of $G_e$ near $\lambda_e^+$ and of the attainability of the minimal speed $c_e^*$.
Some previous works \cite{XLR2021} only considered the case
$\Lambda_e(\lambda)\to+\infty$
as $\lambda\uparrow\lambda_e^+,$
which guarantees the existence of an interior minimizer of $G_e$.
Here we further include the case where
$
\Lambda_e(\lambda_e^+)<+\infty.
$
For instance, consider
$$
\mathcal J(z)
=
C\frac{e^{-|z|}}{1+|z|^q},
\qquad z\in\mathbb R^N,
$$
with $q>N$. Then
$$
\lambda_e^+=1,\
\Lambda_e(1)=\int_{\mathbb R^N}\mathcal J(y)e^{y\cdot e}dy<+\infty \text{ for any } e\in\mathbb S^{N-1}.
$$
Hence, the minimal value of $G_e$ may be achieved at the boundary point
$\lambda_e^+$.
\begin{proof}
By {\rm (J)}, there exists $\mu>0$ such that $\int_{\mathbb R}J_e(z)e^{\mu |z|}dz<+\infty,$
hence $(0,\mu)\subset\mathcal D_e$. Moreover, by H{\"o}lder's inequality, $\mathcal D_e$ is an interval of the form $(0,\lambda_e^+)$ or $(0,\lambda_e^+]$.

For every compact interval $K\subset(0,\lambda_e^+)$, there exist $\lambda_K\in\mathcal D_e$ and $C_K>0$ such that for any fixed $k\geqslant0$ and   $\lambda\in K$,
$|z|^k e^{\lambda z}
\leqslant C_K e^{\lambda_K z}$
for $z\geqslant 0$ and $|z|^k e^{\lambda z}\leqslant C_K$ for $z<0$.
Since $\int_{\mathbb R}J_e(z)e^{\lambda_K z}dz<+\infty$,
differentiation under the integral sign is justified by the dominated convergence theorem. Consequently,
\[
\Lambda_e^{(k)}(\lambda)
=
\int_{\mathbb R}z^kJ_e(z)e^{\lambda z}dz,
\qquad k\geqslant1.
\]
Hence $\Lambda_e,G_e\in C^\infty((0,\lambda_e^+))$. Furthermore,
$\Lambda_e''(\lambda)=\int_{\mathbb R}z^2J_e(z)e^{\lambda z}dz>0$,
so $\Lambda_e$ is strictly convex on $(0,\lambda_e^+)$.

Next,
$$
G_e'(\lambda)
=
\frac{
\lambda\Lambda_e'(\lambda)-\Lambda_e(\lambda)+1-f'(0)
}{\lambda^2}.
$$
Let
$$
H_e(\lambda)
:=
\lambda\Lambda_e'(\lambda)-\Lambda_e(\lambda)+1-f'(0).
$$
Then
$$
H_e'(\lambda)
=
\lambda\Lambda_e''(\lambda)>0, H_e(0)=-f'(0)<0,
$$
hence $H_e$ is strictly increasing on $(0,\lambda_e^+)$. Therefore,
$G_e'$ changes sign at most once on $(0,\lambda_e^+)$, which proves the monotonicity properties of $G_e$.

Since $\Lambda_e(0)=1$ and $f'(0)>0$,
$$
G_e(\lambda)
=
\frac{\Lambda_e(\lambda)-1+f'(0)}{\lambda}
\to+\infty
\qquad\text{as }\lambda\to0^+.
$$
Moreover, $\lim\limits_{\lambda\uparrow\lambda_e^+}\Lambda_e(\lambda)$
exists in $(0,+\infty]$.

If $\lambda_e^+=+\infty$, then, by (J), there exist $\delta>0$ and a measurable set $A\subset\mathbb R^N$ with positive measure such that
\[
y\cdot e\geqslant \delta
\quad\text{for }y\in A,\
\int_A\mathcal J(y)dy>0.
\]
Consequently,
$\Lambda_e(\lambda) \geqslant e^{\delta\lambda}\int_A\mathcal J(y)dy$,
which implies $\Lambda_e(\lambda)\to+\infty$ as $\lambda\to+\infty$.

If $ \lim\limits_{\lambda\uparrow\lambda_e^+}\Lambda_e(\lambda)=+\infty$,
then also
$\lim\limits_{\lambda\uparrow\lambda_e^+}G_e(\lambda)=+\infty$. Combining this with
$\lim\limits_{\lambda\to0^+}G_e(\lambda)=+\infty$,
 and the monotonicity properties of $G_e$, we deduce that $G_e$ admits a unique minimizer $\lambda_e^*\in\mathcal D_e$ such that $c_e^*=G_e(\lambda_e^*)$.

Finally, if
$$
\lambda_e^+<+\infty
\quad\text{and}\quad
\lim_{\lambda\uparrow\lambda_e^+}\Lambda_e(\lambda)<+\infty,
$$
then Fatou's lemma implies $\Lambda_e(\lambda_e^+)<+\infty$,
namely $\lambda_e^+\in\mathcal D_e$. The conclusion then follows from the monotonicity properties of $G_e$.
\end{proof}
The above characterization of $c_e^*$ and the minimizers of $G_e$
will be useful in studying the dependence of traveling fronts on the propagation direction.
The following proposition establishes the corresponding continuity properties.
\begin{proposition}\label{continuity}
Assume conditions (J) and (F) hold. Then the maps $e \mapsto \lambda_e^+$ and $e \mapsto c_e^*$  are continuous from $\mathbb{S}^{N-1}$ to $(0,+\infty]$ and $\mathbb{R}$, respectively.
\end{proposition}
\begin{proof}
\textit{Continuity of the map  $e\mapsto \lambda_e^+$}: To prove the continuity of $e\mapsto \lambda_e^+$, we establish both lower semicontinuity  and upper semicontinuity  on $\mathbb S^{N-1}$.
\begin{itemize}
\item $\liminf\limits_{e_n\to e_0}\lambda_{e_n}^+\geqslant \lambda_{e_0}^+$.  Let
$e_n\to e_0$ in $\mathbb S^{N-1}$. Fix any $\lambda<\lambda_{e_0}^+$ and  choose $\lambda'$ such that
$\lambda<\lambda'<\lambda_{e_0}^+.$ Then $\Lambda_{e_0}(\lambda')<+\infty$. Set
\[
\theta:=\frac{\lambda}{\lambda'}\in(0,1),\
\ell_n:=\frac{\lambda}{1-\theta}(e_n-e_0).
\]
Since $e_n\to e_0$, we have $\ell_n\to0$. Hence, for all sufficiently large $n$,
$|\ell_n|\leqslant\mu$, where $\mu> 0$ is given in (J) and therefore, by (J),
\[
\int_{\mathbb R^N}\mathcal J(y)e^{\ell_n\cdot y}\,dy
\leqslant
\int_{\mathbb R^N}\mathcal J(y)e^{\mu|y|}\,dy<+\infty.
\]
Moreover, $\lambda e_n=\theta\lambda'e_0+(1-\theta)\ell_n.$
By the convexity of $e^x$,
$e^{\lambda y\cdot e_n}
\leqslant
\theta e^{\lambda' y\cdot e_0}
+(1-\theta)e^{\ell_n\cdot y}.
$
Integrating against $\mathcal J$ gives
\[
\Lambda_{e_n}(\lambda)
\leqslant
\theta\Lambda_{e_0}(\lambda')
+(1-\theta)
\int_{\mathbb R^N}\mathcal J(y)e^{\ell_n\cdot y}dy
<+\infty
\]
for all sufficiently large \(n\). Thus
\(\lambda<\lambda_{e_n}^+\) for all sufficiently large \(n\). Since
\(\lambda<\lambda_{e_0}^+\) was arbitrary, $
\liminf_{n\to\infty}\lambda_{e_n}^+
\geqslant\lambda_{e_0}^+.$
The same argument applies when \(\lambda_{e_0}^+=+\infty\), by taking any
\(\lambda>0\).

    \item $\limsup\limits_{e_n\to e_0}\lambda_{e_n}^+\leqslant \lambda_{e_0}^+$.
     Suppose by contradiction that $\limsup\limits_{e_n\to e_0}\lambda_{e_n}^+> \lambda_{e_0}^+.$ Then there exist $\tilde{\lambda} > \lambda_{e_0}^+$ and a subsequence (still denoted by $e_n$) such that $e_n \to e_0$ and $\lambda_{e_n}^+ > \tilde{\lambda}$ for all $n$. Pick any $\lambda_1 \in (\lambda_{e_0}^+, \tilde{\lambda})$, and set $\tau := 1 - \frac{\lambda_1}{\tilde{\lambda}} \in (0,1)$ and $\eta_n := \frac{\lambda_1}{\tau}(e_0-e_n)$. Note that $\eta_n \to 0$ as $n \to \infty$, so $|\eta_n| \leqslant \mu$ for sufficiently large $n$. By the convexity of $x \mapsto e^x$, we have
     \[e^{\lambda_1 (y \cdot e_0)} = e^{(1-\tau)\tilde{\lambda}(y \cdot e_n) + \tau(y \cdot \eta_n)} \leqslant (1-\tau) e^{\tilde{\lambda}(y \cdot e_n)} + \tau e^{y \cdot \eta_n}  \text{ for all } y \in \mathbb{R}^N.\]
     Multiplying by $\mathcal{J}(y)$ and integrating over $\mathbb{R}^N$ yields$$\Lambda_{e_0}(\lambda_1) \leqslant (1-\tau)\Lambda_{e_n}(\tilde{\lambda}) + \tau \int_{\mathbb{R}^N} \mathcal{J}(y) e^{y \cdot \eta_n} dy < +\infty,$$which contradicts $\lambda_1 > \lambda_{e_0}^+$ (i.e., $\Lambda_{e_0}(\lambda_1) = +\infty$).

 \end{itemize}

\textit{Continuity of the map $e \mapsto c_e^*$}:
We next show that the minimizers of $G_e$ can be localized in a
compact interval independently of $e\in\mathbb S^{N-1}$.
Extend $G_e$ by setting $G_e(\lambda)=+\infty$ whenever
$\lambda\notin\mathcal D_e$. Fix $\lambda_0\in(0,\mu)$. Since $e^{\lambda_0 x\cdot e}\leqslant e^{\lambda_0|x|}
\leqslant e^{\mu|x|},$
assumption {\rm (J)} yields
$c_0:=\sup_{e\in\mathbb S^{N-1}}G_e(\lambda_0)
<+\infty.$
In particular,
$
c_e^*\leqslant G_e(\lambda_0)\leqslant c_0$ for every $e\in\mathbb S^{N-1}$. We first obtain a uniform lower bound for small $\lambda$.
Set
$M_1:=\int_{\mathbb R^N}|x|\mathcal J(x)\,dx<+\infty.$
Using $e^s\geqslant1+s$, we have
\[
\Lambda_e(\lambda)
=
\int_{\mathbb R^N}
\mathcal J(x)e^{\lambda x\cdot e}\,dx
\geqslant
1+\lambda\int_{\mathbb R^N}
(x\cdot e)\mathcal J(x)\,dx
\geqslant
1-\lambda M_1.
\]
Consequently,
$$
G_e(\lambda)
=
\frac{\Lambda_e(\lambda)-1+f'(0)}{\lambda}
\geqslant
\frac{f'(0)}{\lambda}-M_1.
$$
Since $f'(0)>0$, we may choose $\delta\in(0,\lambda_0)$,
independently of $e$, such that
$\frac{f'(0)}{\delta}-M_1>c_0$.
It follows that
$
G_e(\lambda)>c_0\geqslant c_e^*$
for all $e\in\mathbb S^{N-1}$ and $0<\lambda<\delta$.

We now derive a uniform  estimate for large $\lambda$.
For $e\in\mathbb S^{N-1}$, set
$A_e:=\left\{x\in B_{r_0}:
x\cdot e\geqslant\frac{r_0}{2}
\right\}.$
By rotational invariance of the Lebesgue measure,
$|A_e|=:a_0>0$ is independent of $e$. Since
$\mathcal J\geqslant\kappa_0$ a.e. in $B_{r_0}$,
we obtain, for every $\lambda>0$,
\[\Lambda_e(\lambda)\geqslant
\int_{A_e} \mathcal J(x)e^{\lambda x\cdot e}\,dx\geqslant
\kappa_0a_0e^{r_0\lambda/2}\].
Hence
$$
G_e(\lambda)
\geqslant
\frac{
\kappa_0a_0e^{r_0\lambda/2}
-1+f'(0)
}{\lambda}.
$$
The right-hand side tends to $+\infty$ as
$\lambda\to+\infty$, independently of
$e\in\mathbb S^{N-1}$. Therefore, there exists
$M>\lambda_0$, independent of $e$, such that
$G_e(\lambda)>c_0\geqslant c_e^*$
for every $e\in\mathbb S^{N-1}$ and $\lambda>M$.

We have thus shown that every minimizer of $G_e$ belongs to the
fixed compact interval $[\delta,M]$. In particular,
$c_e^*=\min_{\lambda\in[\delta,M]}G_e(\lambda)$,
$e\in\mathbb S^{N-1}$,
and any minimizer $\lambda_e^*$ satisfies
$\delta\leqslant\lambda_e^*\leqslant M$.

\begin{itemize}
\item $\limsup\limits_{n \to \infty} c_{e_n}^* \leqslant c_{e_0}^*$.
Let $\lambda_0 \in [\delta, \lambda_{e_0}^+]$ be a minimizer of $G_{e_0}$ and $e_n\to e_0$ as $n\to+\infty$, so $c_{e_0}^* = G_{e_0}(\lambda_0)$. We consider two cases:
\begin{itemize}
\item Case 1: $\lambda_0 < \lambda_{e_0}^+$. Since $\lambda_e^+$ is continuous in $e$, $\lambda_0 < \lambda_{e_n}^+$ for sufficiently large $n$. Thus $\lambda_0 \in (0, \lambda_{e_n}^+)$, and by an argument similar to the one used above to prove the continuity of $\Lambda_e(\lambda)$ with respect to $e$, $G_{e_n}(\lambda_0) \to G_{e_0}(\lambda_0)$. Therefore:$$\limsup_{n \to \infty} c_{e_n}^* \leqslant \limsup_{n \to \infty} G_{e_n}(\lambda_0) = G_{e_0}(\lambda_0) = c_{e_0}^*.$$
    \item Case 2: $\lambda_0=\lambda_{e_0}^+$. $G_{e_0}$ is strictly decreasing on $\mathcal D_{e_0}=(0,\lambda_{e_0}^+]$. For any small $\epsilon > 0$, set $\lambda_\epsilon:=\lambda_{e_0}^+ -\epsilon< \lambda_{e_0}^+$. For large $n$, $\lambda_\epsilon<\lambda_{e_n}^+$. Hence $c_{e_n}^*\leqslant G_{e_n}(\lambda_\epsilon)$. By the argument above, $G_{e_n}(\lambda_\epsilon) \to G_{e_0}(\lambda_\epsilon)$. Taking the upper limit as $n \to \infty$ gives
        \[\limsup_{n\to\infty}c_{e_n}^* \leqslant \lim_{n \to \infty} G_{e_n}(\lambda_\epsilon) = G_{e_0}(\lambda_\epsilon).\]
        Taking $\epsilon \to 0^+$, we have $\limsup\limits_{n\to\infty}c_{e_n}^* \leqslant \lim\limits_{\lambda\uparrow \lambda_{e_0}^+}G_{e_0}(\lambda)=c_{e_0}^*$.
        \end{itemize}
        Thus we have $\limsup_{n \to \infty} c_{e_n}^* \leqslant c_{e_0}^*$.
 \item $\liminf\limits_{n \to \infty} c_{e_n}^* \geqslant c_{e_0}^*$. For each $n$, let $\lambda_n \in [\delta, \lambda_{e_n}^+]$ be a minimizer of $G_{e_n}$, so $c_{e_n}^* = G_{e_n}(\lambda_n)$. By the uniform localization of the minimizers established above,
$\lambda_n\in[\delta,M]$ for every $n$. Hence, up to a
subsequence, $\lambda_n\to\lambda_1\in[\delta,M]$.  By Fatou's Lemma\[\liminf_{n \to \infty} \Lambda_{e_n}(\lambda_n) = \liminf_{n \to \infty} \int_{\mathbb{R}^N} \mathcal{J}(y) e^{\lambda_n (y \cdot e_n)} dy \geqslant \int_{\mathbb{R}^N} \mathcal{J}(y) e^{\lambda_1 (y \cdot e_0)} dy = \Lambda_{e_0}(\lambda_1).\]
     Since $\lambda_n \to \lambda_1 \geqslant \delta > 0$, this lower semicontinuity of $\Lambda_e$ passes directly to $G_e$
     \[\liminf_{n \to \infty} c_{e_n}^* = \liminf_{n \to \infty} G_{e_n}(\lambda_n) \geqslant G_{e_0}(\lambda_1) \geqslant \min_{\lambda \in [\delta, \lambda_{e_0}^+]} G_{e_0}(\lambda) = c_{e_0}^*.\]
\end{itemize}
  We complete the proof of the continuity of $e \mapsto c_e^*$ on $\mathbb{S}^{N-1}$.
\end{proof}

We now recall classical results on the long-time behavior of solutions to~\eqref{equation} with compactly supported initial data $u_0\geqslant 0$, $u_0\not \equiv0$.
Let $\mathcal S$ and the  speed $s(e)$ be defined as in~\eqref{se}.
\begin{proof}[Proof of Proposition~\ref{NewS}]
We first prove (i). The convexity of $\mathcal S$ follows directly
from its definition, since $\mathcal S$ is an intersection of
half-spaces. Moreover, by the uniform estimate established above,
$c_e^*\leqslant c_0$ for every $e\in\mathbb S^{N-1}$. Hence, if
$x\in\mathcal S\setminus\{0\}$, taking $e=\hat x$ gives
$|x|<c_{\hat x}^*\leqslant c_0$, and therefore $\mathcal S$ is
bounded.

The openness of $\mathcal S$ follows from the continuity of
$e\mapsto c_e^*$. Indeed, for any $x\in\mathcal S$, the function
$e\mapsto c_e^*-x\cdot e$ is positive and continuous on the
compact sphere $\mathbb S^{N-1}$. Thus
$\delta_x:=\min_{e\in\mathbb S^{N-1}}(c_e^*-x\cdot e)>0$.
If $|y-x|<\delta_x/2$, then
$y\cdot e\leqslant x\cdot e+|y-x|<c_e^*$ for every
$e\in\mathbb S^{N-1}$. Hence
$B_{\delta_x/2}(x)\subset\mathcal S$, and $\mathcal S$ is open.

By assumption {\rm (J)},
$\int_{\mathbb R^N}|x|\mathcal J(x)\,dx<+\infty$, so that
$\mathbf p=\int_{\mathbb R^N}x\mathcal J(x)\,dx$ is well defined.
Let $\lambda_e^*$ be a minimizer of $G_e$. By the uniform
localization established above, there exist $0<\delta<M<+\infty$,
independent of $e$, such that $\lambda_e^*\in[\delta,M]$. Since
$e^s\geqslant1+s$, we have
$\Lambda_e(\lambda)\geqslant1+\lambda\mathbf p\cdot e$ for every
$\lambda>0$, and consequently
\[
G_e(\lambda)-\mathbf p\cdot e
\geqslant\frac{f'(0)}{\lambda}.
\]
Evaluating this inequality at $\lambda=\lambda_e^*$ yields
\[
c_e^*-\mathbf p\cdot e
\geqslant\frac{f'(0)}{\lambda_e^*}
\geqslant\frac{f'(0)}{M}
=:\gamma>0,\ e\in\mathbb S^{N-1}.
\]
Therefore, if $|x-\mathbf p|<\gamma$, then
$x\cdot e\leqslant\mathbf p\cdot e+|x-\mathbf p|<c_e^*$ for every
$e\in\mathbb S^{N-1}$. Hence
$B_\gamma(\mathbf p)\subset\mathcal S$, which proves that
$\mathbf p\in\operatorname{int}\mathcal S$ and, in particular,
that $\mathcal S$ is nonempty. Since $\mathcal S$ is convex, it is
also star-shaped with respect to $\mathbf p$.

We next prove (ii). Fix $e\in\mathbb S^{N-1}$ and suppose first
that $re\in\mathcal S$. Since $\mathcal S$ is open, there exists
$\varepsilon>0$ such that both $(r-\varepsilon)e$ and
$(r+\varepsilon)e$ belong to $\mathcal S$. For every $\xi$ with
$\xi\cdot e>0$, we then have
$(r+\varepsilon)\xi\cdot e<c_\xi^*$, and hence
$r+\varepsilon\leqslant s(e)$. Thus $r<s(e)$. Similarly, for
$\xi\cdot e<0$, the inequality
$(r-\varepsilon)\xi\cdot e<c_\xi^*$ gives
\[
r-\varepsilon\geqslant\sup_{\xi\cdot e<0}\frac{c_\xi^*}{\xi\cdot e}=-s(-e),
\]
and therefore $-s(-e)<r$. Finally, if $\xi\cdot e=0$, then
$0=(re)\cdot\xi<c_\xi^*$, so necessarily $c_\xi^*>0$.

Conversely, assume that $-s(-e)<r<s(e)$ and that $c_\xi^*>0$
whenever $\xi\cdot e=0$. If $\xi\cdot e>0$, then
$r<s(e)\leqslant c_\xi^*/(\xi\cdot e)$, which gives
$re\cdot\xi<c_\xi^*$. If $\xi\cdot e<0$, then
$r>-s(-e)\geqslant c_\xi^*/(\xi\cdot e)$; since
$\xi\cdot e<0$, multiplying by $\xi\cdot e$ reverses the
inequality and again gives $re\cdot\xi<c_\xi^*$. Finally, if
$\xi\cdot e=0$, then $re\cdot\xi=0<c_\xi^*$. Hence
$re\in\mathcal S$, and \eqref{origin-ray-section} follows.

It remains to prove (iii). For $e\in\mathbb S^{N-1}$, set
\[\widetilde\sigma(e):=
\inf_{\xi\cdot e>0}\frac{c_\xi^*-\mathbf p\cdot\xi}{\xi\cdot e}.
\]
From the estimate
$c_\xi^*-\mathbf p\cdot\xi\geqslant\gamma$, we obtain
$\widetilde\sigma(e)\geqslant\gamma>0$, while taking $\xi=e$
gives
$\widetilde\sigma(e)\leqslant c_e^*-\mathbf p\cdot e<+\infty$.
Thus $\widetilde\sigma(e)\in(0,+\infty)$.

We claim that $\sigma(e)=\widetilde\sigma(e)$. Let
$0\leqslant r<\widetilde\sigma(e)$. If $\xi\cdot e>0$, then
$r<\frac{c_\xi^*-\mathbf p\cdot\xi}{ \xi\cdot e}$, so
$(\mathbf p+re)\cdot\xi<c_\xi^*$. If $\xi\cdot e\leqslant0$,
then
$(\mathbf p+re)\cdot\xi\leqslant\mathbf p\cdot\xi<c_\xi^*$,
because $\mathbf p\in\mathcal S$. Therefore
$\mathbf p+re\in\mathcal S$, and hence
$\sigma(e)\geqslant\widetilde\sigma(e)$.

Conversely, if $r>\widetilde\sigma(e)$, then by the definition of
the infimum there exists $\xi$ with $\xi\cdot e>0$ such that
$(c_\xi^*-\mathbf p\cdot\xi)/(\xi\cdot e)<r$. Thus
$(\mathbf p+re)\cdot\xi>c_\xi^*$, and consequently
$\mathbf p+re\notin\mathcal S$. Hence
$\sigma(e)\leqslant\widetilde\sigma(e)$, proving \eqref{sigma-FG}.

Finally, if $x\in\mathcal S\setminus\{\mathbf p\}$, write
$x=\mathbf p+re$, where
$r=|x-\mathbf p|>0$ and
$e=(x-\mathbf p)/|x-\mathbf p|$. Since $\mathcal S$ is open,
$r<\sigma(e)$. Conversely, the preceding argument shows that
$\mathbf p+re\in\mathcal S$ whenever $0\leqslant r<\sigma(e)$.
Therefore,
$\mathcal S=\left\{
\mathbf p+re:e\in\mathbb S^{N-1},\ 0\leqslant r<\sigma(e)\right\}.$
This completes the proof.
\end{proof}

\section{Proofs}\label{proofs}
\subsection{Compactly Supported Initial Data}
We first consider solutions arising from compactly supported initial data.
Our proof follows the geometric approach introduced by Rossi \cite{R2017}, with modifications to account for the  asymmetric  dispersal kernels.
In the present setting, the natural center of propagation is given by $\mathbf p=\int_{\mathbb R^N}x\mathcal J(x)dx,$
which does not necessarily coincide with the origin.

Before carrying out the geometric argument, we recall the following local invasion property along the trajectory $t\mathbf p$, which is a direct consequence of the hair-trigger effect in \cite[Theorem 2.5]{FT2018}.

\begin{lemma}\label{pinvasion}
Assume that (J) and (F) hold, and let $u$ be the solution of \eqref{equation}
corresponding to a nontrivial compactly supported initial datum
$0\leqslant u_0\leqslant 1$. Assume, in addition, that $u_0$ is bounded from below
by a positive constant on some ball. Then
\begin{equation}\label{pinvasion-1}
\lim_{t\to+\infty}\operatorname{ess\,inf}_{x\in K}u(t,t\mathbf p+x)= 1 \text{ for every compact  } K\subset\mathbb R^N.
\end{equation}
\end{lemma}
In the proof of theorems, Lemma~\ref{pinvasion} will only be applied to continuous compactly supported data. Accordingly, in that application the essential infimum in
\eqref{pinvasion-1} may be replaced by the usual infimum.

Lemma \ref{pinvasion} provides only local invasion in bounded neighborhoods of  $t\mathbf p$. In particular, it does not contain any information about the
shape of the spreading set at the linear scale $O(t)$. Its role below is only to ensure
that arbitrarily large bounded dilations of $\mathcal O$ are eventually invaded
around $t\mathbf p$. The linear-scale spreading described by $\mathcal S$ will be obtained
independently through the geometric argument below.
\begin{proof}[Proof of Theorem \ref{invasionS}]
Notice that  there exists $\alpha \in(0,1)$ such that $f$  is  nonincreasing in $[\alpha,1]$ and  $f(u)>0$  for $u\in(\alpha,1)$.
Here we also extend  $f(u)$ to $0$ for $u\notin[0,1]$.  This extension is globally Lipschitz on $\mathbb R$. Set $\mathbb R_-:=(-\infty,0]$. For a fixed $\zeta\in\mathbb R^N$, consider
\begin{equation}\label{-R}
v_t =
\mathcal J\star v-v+\zeta\cdot\nabla v+f(v),\ t<0,\ x\in\mathbb R^N.
\end{equation}
We use the following  viscosity notion on
$\mathbb R_-\times\mathbb R^N$. A bounded lower semicontinuous function
$\overline v:\mathbb R_-\times\mathbb R^N\to\mathbb R$
is called a viscosity supersolution of \eqref{-R} if, for every $(t_0,x_0)\in\mathbb R_-\times\mathbb R^N$ and every function $\varphi\in C^1$ defined in a neighborhood of $(t_0,x_0)$ such that
$\varphi$ touches $\overline  v$ from below at $(t_0,x_0)$, namely,
$ \overline  v(t_0,x_0)=\varphi(t_0,x_0)$
and
$\overline  v(t,x)\geqslant\varphi(t,x)$ for $(t,x)$ in a neighborhood of $(t_0,x_0)$ relative to $\mathbb R_-\times\mathbb R^N$, one has
\begin{equation}\label{vis-super}
\varphi_t(t_0,x_0) - (\mathcal J\star \overline v)(t_0,x_0)+\overline v(t_0,x_0)-\zeta\cdot\nabla\varphi(t_0,x_0)-f(\overline v(t_0,x_0))\geqslant0.
\end{equation}
Similarly, a bounded upper semicontinuous function
$\underline v:\mathbb R_-\times\mathbb R^N\to\mathbb R$
is called a viscosity subsolution of \eqref{-R} if, whenever
$\varphi$ touches $\underline v$ from above at $(t_0,x_0)$, one has
\begin{equation}\label{vis-sub}
\varphi_t(t_0,x_0)-(\mathcal J\star \underline v)(t_0,x_0)+\underline v(t_0,x_0)-\zeta\cdot\nabla\varphi(t_0,x_0)
-f(\underline v(t_0,x_0))\leqslant0.
\end{equation}
At $t_0=0$, the touching condition is understood relative to $\mathbb R_-\times\mathbb R^N$.

 Before proving the auxiliary lemmas, we record a finite-time
comparison principle for viscosity sub- and supersolutions of
\eqref{-R}. Its proof is based on the standard
doubling-of-variables argument for integro-differential equations; see, for instance, \cite{AT1996}. We include the details since they are particularly simple for the present convolution operator. The proof will be given at the end of this section.
\begin{proposition}\label{finite-CP}
Assume that (J) and (F) hold. Fix $\zeta\in\mathbb R^N$,
$T_1<T_2\leqslant 0$, $e\in\mathbb S^{N-1}$ and $K\in\mathbb R$. Let $\underline v$ and $\overline v$ be bounded functions on $[T_1,T_2]\times\mathbb R^N$, with $\underline v$ upper semicontinuous and $\overline v$ lower semicontinuous.
Assume that they are respectively a viscosity subsolution and a viscosity supersolution of \eqref{-R} in
$(T_1,T_2]\times\{x\in\mathbb R^N:x\cdot e<K\}.$
At $t=T_2$, the viscosity inequalities are understood in the one-sided sense relative to $(-\infty,T_2]\times\mathbb R^N$. Assume moreover that
$\underline v(T_1,x)\leqslant\overline v(T_1,x)$ for every $x\in\mathbb R^N$, and
\begin{equation}\label{finite-exterior}
\underline v(t,x)\leqslant \overline v(t,x)
\text{ for } T_1\leqslant t\leqslant T_2,\quad x\cdot e\geqslant K.
\end{equation}
Then
\[\underline v(t,x)\leqslant
\overline v(t,x)\text{ for all }(t,x)\in[T_1,T_2]\times\mathbb R^N.
\]
\end{proposition}

\begin{lemma}\label{mathcalH}
Assume that  (J) and (F) hold, and fix $\zeta\in\mathbb R^N$.
Let $\overline v:\mathbb R_-\times\mathbb R^N\to\mathbb R
$ be a bounded lower semicontinuous viscosity supersolution of \eqref{-R}. Assume that there exists a set
$\mathcal H\subset\mathbb R^N$ such that
\[
\sup_{x\in\mathcal H}
\operatorname{dist}
\bigl(x,\mathbb R^N\backslash \mathcal H\bigr)
=+\infty,\
\inf_{t\leqslant 0,\ x\in\mathcal H}\overline v(t,x)>\alpha.
\]
Then
\[
\liminf_{
\operatorname{dist}(x,\mathbb R^N\setminus\mathcal H)\to+\infty}
\left(
\inf_{t\leqslant 0}\overline v(t,x)
\right)
\geqslant1.
\]
\end{lemma}

\begin{proof}
Assume by contradiction that
\[
h:= \liminf_{\operatorname{dist}(x,\mathbb R^N\setminus\mathcal H)\to+\infty}\left(\inf_{t\leqslant 0}\overline v(t,x)\right)<1.
\]
Since $
\inf\limits_{t\leqslant0,x\in\mathcal H }
\overline v(t,x)>\alpha$,
we have $\alpha<h<1$.
Hence, by (F), $f(h)>0.$ By the definition of $h$, there exist $x_n\in\mathcal H$ and $t_n\leqslant0$ such that
$d_n:=\operatorname{dist}\bigl(x_n,\mathbb R^N\setminus\mathcal H\bigr) \to+\infty $
and $
\overline v(t_n,x_n)\to h$ as $n\to+\infty$.
Set
$v_n(t,x):=\overline v(t+t_n,x+x_n)$ for $
(t+t_n,x)\in\mathbb R_-\times\mathbb R^N.
$
Each $v_n$ is again a bounded lower semicontinuous viscosity
supersolution of \eqref{-R}.

Following \cite{BI2008,AT1996}, define the half-relaxed limit
\[v_\infty(t,x)
:=
\liminf\limits_{\substack{n\to\infty\\(s,y)\to(t,x)}}
v_n(s,y).\] Then $v_\infty$ is lower semicontinuous. Moreover,
$v_\infty$ is a viscosity supersolution of \eqref{-R}.
 Indeed, the passage to the limit in the nonlocal
term follows directly from the positivity of $\mathcal J$ and Fatou's lemma: if $(s_n,y_n)\to(t,x)$, then
$v_\infty(t,x-z)\leqslant\liminf\limits_{n\to\infty}v_n(s_n,y_n-z)$
for every $z\in\mathbb R^N$, and consequently
\[(\mathcal J\star v_\infty)(t,x) \leqslant \liminf_{n\to\infty} (\mathcal J\star v_n)(s_n,y_n).\]

We claim that
$v_\infty(t,x)\geqslant h
$  for all fixed $
(t,x)\in\mathbb R_-\times\mathbb R^N$.
To prove this, let $(s_n,y_n)\to(t,x)$ with $s_n\leqslant0$. Since $(y_n)$ is bounded, there exists $C>0$ such that $|y_n|\leqslant C$. Since the distance function is $1$-Lipschitz, we have
\[
\operatorname{dist}
\bigl(x_n+y_n,\mathbb R^N\backslash \mathcal H\bigr)
\geqslant
\operatorname{dist}
\bigl(x_n,\mathbb R^N\setminus\mathcal H\bigr)
-|y_n|
\geqslant d_n-C
\to +\infty.\]
Moreover, $t_n+s_n\leqslant0.$
Therefore, by the definition of $h$,
$\liminf_{n\to\infty} \overline v(t_n+s_n,x_n+y_n) \geqslant h.$
Since
$v_n(s_n,y_n) =\overline v(t_n+s_n,x_n+y_n)$,
we obtain
$\liminf_{n\to\infty}v_n(s_n,y_n)\geqslant h$.
As the sequence $(s_n,y_n)\to(t,x)$ was arbitrary,
 $v_\infty(t,x)\geqslant h$  follows.

On the other hand,
\[
v_\infty(0,0)\leqslant\liminf_{n\to\infty}v_n(0,0)=\lim_{n\to\infty}\overline v(t_n,x_n)=h.
\]
This  yields\[
v_\infty(0,0)=h,\ v_\infty(t,x)\geqslant h\text{ for all }
(t,x)\in\mathbb R_-\times\mathbb R^N.
\]
Hence the constant function
$\varphi(t,x)\equiv h$
touches $v_\infty$ from below at $(0,0)$ in the one-sided
sense relative to $\mathbb R_-\times\mathbb R^N$.
The viscosity supersolution inequality gives
$-(\mathcal J\star v_\infty)(0,0)+h-f(h)\geqslant0.
$
Consequently,
$$
0
\leqslant
-(\mathcal J\star v_\infty)(0,0)+h-f(h)
\leqslant
-f(h)<0,
$$
where the last inequality follows from $h\in(\alpha,1)$
and hence $f(h)>0$.
This contradiction proves the result.
\end{proof}
The second auxiliary result is a comparison principle for viscosity
sub- and supersolutions.

\begin{lemma}\label{CP}
Assume that (J) and (F) hold, and fix $\zeta\in\mathbb R^N$.
Let $\underline v $ and $\overline v$
be bounded functions such that $\underline v$ is an upper
semicontinuous viscosity subsolution of \eqref{-R} and
$\overline v$ is a lower semicontinuous viscosity supersolution of \eqref{-R}. Suppose that, for some $e\in\mathbb S^{N-1}$,
\begin{equation}\label{CP-1}
\overline v\geqslant0,\
\liminf_{x\cdot e\to-\infty}
\overline v(t,x)
\geqslant 1 \text{ uniformly in }t\leqslant0.
\end{equation}
Assume moreover that
$\underline v\leqslant1$
and that there exist $\gamma>0$ and $L\in\mathbb R$ such that
\begin{equation}\label{CP-2}
\underline v(t,x)\leqslant0
\text{ for }t\leqslant0,\
x\cdot e\geqslant\gamma t+L.
\end{equation}
Then
\[
\underline v(t,x)\leqslant\overline v(t,x)
\text{ for all }
(t,x)\in\mathbb R_-\times\mathbb R^N.\]
\end{lemma}
We point out that the same proof applies if the reference level $0$ in the inequalities $\overline v\geqslant0$ and
$\underline v\leqslant0$ is replaced by any constant $a<1$.
\begin{proof}
Choose $\beta\in(\alpha,1)$. By \eqref{CP-1}, there exists
$K_0>0$ such that $
\overline v(t,x)\geqslant\beta
\text{ for }t<0,\ x\cdot e\leqslant-K_0$.
Since $\gamma>0$, choose $T_0<0$ such that
$\gamma T_0+L\leqslant-K_0$. Fix $\varepsilon>0$. By \eqref{CP-1}, there exists
$K_\varepsilon>0$ such that
$
\overline v(t,x)
\geqslant1-\frac{\varepsilon}{2} $
for $t<0 $  and $x\cdot e\leqslant-K_\varepsilon.
$
Choose $T_\varepsilon<T_0$ sufficiently negative so that
$
\gamma T_\varepsilon+L\leqslant-K_\varepsilon$. We may clearly choose $T_\varepsilon$ in such a way that
$T_\varepsilon\to-\infty$ as $\varepsilon\to0$.

We first compare $\underline v$ and $\overline v$ on
$[T_\varepsilon,T_0]$.
Set $\overline v^\varepsilon(t,x):=\overline v(t,x)+\varepsilon$.
At $t=T_\varepsilon$,
$\underline v(T_\varepsilon,x)
< \overline v^\varepsilon(T_\varepsilon,x)$
 for every $x\in\mathbb R^N$. Indeed, if
$
x\cdot e\geqslant\gamma T_\varepsilon+L,
$
then  $\overline v\geqslant0$ imply
$
\underline v(T_\varepsilon,x)
\leqslant0 <\overline v(T_\varepsilon,x)+\varepsilon$.
If
$
x\cdot e<\gamma T_\varepsilon+L, $
then we have
$\overline v^\varepsilon(T_\varepsilon,x)
\geqslant1+\frac{\varepsilon}{2}>1\geqslant\underline v(T_\varepsilon,x)$.

We next show that $\overline v^\varepsilon$ is a viscosity
supersolution of \eqref{-R} in
\[
Q_\varepsilon :=(T_\varepsilon,T_0]\times\{x\in\mathbb R^N:x\cdot e<-K_0\}.
\]
Let $\varphi$ touch $\overline v^\varepsilon$ from below at
$(t_0,x_0)\in Q_\varepsilon$. Then
$\varphi-\varepsilon$ touches $\overline v$ from below at the
same point. Hence
\[\varphi_t
-\mathcal J\star\overline v^\varepsilon
+\overline v^\varepsilon
-\zeta\cdot\nabla\varphi
-f(\overline v^\varepsilon)
\geqslant
f(\overline v)-f(\overline v+\varepsilon)
\]
at $(t_0,x_0)$.
Note that  $(t_0,x_0)\in Q_\varepsilon$,
$\overline v(t_0,x_0)\geqslant\beta>\alpha.$
If $\overline v(t_0,x_0)\leqslant1$, the monotonicity of $f$
on $[\alpha,1]$ and the extension $f=0$ on $(1,+\infty)$ give
$
f(\overline v+\varepsilon)
\leqslant f(\overline v). $
If $\overline v(t_0,x_0)>1$, both terms are equal to $0$.
Thus
$
f(\overline v)-f(\overline v+\varepsilon)\geqslant0,
$
and $\overline v^\varepsilon$ is a viscosity supersolution in $Q_\varepsilon$.

We also have the exterior ordering. Indeed, for
$t\in[T_\varepsilon,T_0]$ and $x\cdot e\geqslant-K_0$,
here $
x\cdot e\geqslant-K_0\geqslant\gamma T_0+L \geqslant\gamma t+L$.
Hence, by \eqref{CP-2},
$$
\underline v(t,x)
\leqslant0
<
\overline v^\varepsilon(t,x).
$$
Together with the case $t=T_\varepsilon$,  Proposition \ref{finite-CP} in $Q_\varepsilon$ yields
$\underline v(t,x)
\leqslant
\overline v(t,x)+\varepsilon$
for
$T_\varepsilon\leqslant t\leqslant T_0$ and  $x\in\mathbb R^N$.

We now propagate the comparison from $T_0$ to $0$.
Define $\delta_\varepsilon(t):=\varepsilon e^{l(t-T_0)}$ for $T_0\leqslant t\leqslant0$, where $l$ is a Lipschitz constant of $f$, and set
$w^\varepsilon(t,x):=\overline v(t,x)+\delta_\varepsilon(t)$.
We claim that $w^\varepsilon$ is a viscosity supersolution of \eqref{-R} on $(T_0,0]\times\mathbb R^N$, where the viscosity inequality at $t=0$ is understood in the one-sided sense. Indeed, let $\varphi$ touch $w^\varepsilon$ from below at
$(t_0,x_0)$. Then $\psi(t,x):=\varphi(t,x)-\delta_\varepsilon(t)$ touches
$\overline v$ from below at the same point. Hence
\[
\psi_t-\mathcal J\star\overline v+\overline v
-\zeta\cdot\nabla\psi-f(\overline v)\geqslant0.
\]
Since $\psi_t=\varphi_t-\delta_\varepsilon'$ and
$\nabla\psi=\nabla\varphi$, while
$-\mathcal J\star w^\varepsilon+w^\varepsilon
=-\mathcal J\star\overline v+\overline v$, we obtain
\[
\varphi_t-\mathcal J\star w^\varepsilon+w^\varepsilon
-\zeta\cdot\nabla\varphi-f(w^\varepsilon)
\geqslant
\delta_\varepsilon'
+f(\overline v)-f(\overline v+\delta_\varepsilon).
\]
Since $f$ is $l$-Lipschitz and
$\delta_\varepsilon'=l\delta_\varepsilon$, the right-hand side
is nonnegative. Thus $w^\varepsilon$ is a viscosity
supersolution.

By the comparison already obtained on $[T_\varepsilon,T_0]$,
\[
\underline v(T_0,x)
\leqslant\overline v(T_0,x)+\varepsilon
=w^\varepsilon(T_0,x)
\text{ for all }x\in\mathbb R^N.
\]
Moreover, for $T_0\leqslant t\leqslant0$ and
$x\cdot e\geqslant L$, we have
$x\cdot e\geqslant L\geqslant\gamma t+L$. Hence, by
\eqref{CP-2} and \eqref{CP-1},
\[
\underline v(t,x)\leqslant0
<\overline v(t,x)+\delta_\varepsilon(t)
=w^\varepsilon(t,x).
\]
Therefore all the assumptions of Proposition~\ref{finite-CP}
are satisfied on $[T_0,0]$, with $K=L$. It follows that
\[
\underline v(t,x)
\leqslant
\overline v(t,x)+\varepsilon e^{l(t-T_0)}
\]
for $T_0\leqslant t\leqslant0$ and $x\in\mathbb R^N$.
Letting $\varepsilon\to0^+$ gives
$\underline v\leqslant\overline v$ on
$[T_0,0]\times\mathbb R^N$.

Finally, fix $t<T_0$. Since $T_\varepsilon\to-\infty$ as
$\varepsilon\to0^+$, one has $T_\varepsilon<t$ for all
sufficiently small $\varepsilon>0$. The comparison on
$[T_\varepsilon,T_0]$ then gives
$\underline v(t,x)\leqslant\overline v(t,x)+\varepsilon$.
Letting $\varepsilon\to0^+$ yields
$\underline v(t,x)\leqslant\overline v(t,x)$.
Thus
$\underline v\leqslant\overline v$
on $\mathbb R_-\times\mathbb R^N$.
\end{proof}

The proof of the first part of \eqref{definitionw} will first be
carried out under the additional assumption
\begin{equation}\label{additional}
\mathcal S \text{ satisfies the uniform interior ball condition}.
\end{equation}
This assumption will be removed at the end of the proof.

Let $u$ be the solution under consideration, with compactly
supported initial datum $u_0$. By the assumption on $u_0$, there exist $x_0\in\mathbb R^N$, $r_0>0$ and $\rho_0>0$ such that $u_0\geqslant\rho_0$ a.e. in $B_{r_0}(x_0)$.
Choose $\chi\in C_c(\mathbb R^N)$ such that
$0\leqslant\chi\leqslant1$,
$\chi\equiv1$ on $B_{r_0/2}(x_0)$, and
$\operatorname{supp}\chi\subset B_{r_0}(x_0)$, and set
$\underline u_0:=\rho_0\chi.$
Then $\underline u_0\in C_c(\mathbb R^N)$ and
$0\leqslant\underline u_0\leqslant u_0$ a.e. in $\mathbb R^N$. Let $\underline u$ be the solution corresponding to
$\underline u_0$. By the comparison principle,
$
\underline u(t,x)\leqslant u(t,x)$ for $t\geqslant0$.
It is therefore sufficient to establish the desired interior propagation property for $\underline u$. Thus, in what follows, we may replace $u$ by $\underline u$ and assume, without loss of generality for the  argument, that $u_0\in C_c(\mathbb R^N)$.

Recall that $\mathcal O=\mathcal S-\mathbf p$ and
$0\in\operatorname{int}\mathcal O$. Since $\mathcal S$ is convex,
so is $\mathcal O$, and hence
$r_1\mathcal O\subset r_2\mathcal O$ whenever
$0<r_1<r_2$. Fix $\eta\in(0,1)$. For $t\geqslant0$, define
\begin{equation}\label{definition-R}
\mathcal R^\eta(t)
:=
\sup\left\{
r\geqslant0:
u(t,t\mathbf p+x)>\eta
\text{ for every }x\in r\mathcal O
\right\}.
\end{equation}
The quantity $\mathcal R^\eta(t)$ is finite for every $t>0$,
since $u(t,x)\to0$ as $|x|\to+\infty$ for each fixed $t>0$.
On the other hand, for every fixed $r>0$, the set
$r\overline{\mathcal O}$ is compact. Since the present initial
datum belongs to $C_c(\mathbb R^N)$, the continuous-data version
of Lemma~\ref{pinvasion} gives
\[
\lim_{t\to+\infty}
\inf_{x\in r\overline{\mathcal O}}
u(t,t\mathbf p+x)=1.
\]
Hence, for all sufficiently large $t$,
\[
u(t,t\mathbf p+x)>\eta
\text{ for every }x\in r\overline{\mathcal O}.
\]
In particular,
$\mathcal R^\eta(t)\geqslant r$ for all sufficiently large $t$.
Since $r>0$ is arbitrary, we conclude that
\begin{equation}\label{R-infinity}
\mathcal R^\eta(t)\longrightarrow+\infty
\text{ as }t\to+\infty.
\end{equation}
We also record that, for every fixed $\eta\in(0,1)$, the map
$t\mapsto\mathcal R^\eta(t)$ is lower semicontinuous.
Fix $t_0\geqslant0$ and $r<\mathcal R^\eta(t_0)$. Choose
$r<r_1<r_2<\mathcal R^\eta(t_0)$. Since
$0\in\operatorname{int}\mathcal O$ and $\mathcal O$ is convex,
we have
$r_1\overline{\mathcal O}\subset r_2\mathcal O$. Hence
$\min_{x\in r_1\overline{\mathcal O}}
u(t_0,t_0\mathbf p+x)>\eta$. By the joint continuity of $u(t,x)$ and the compactness of
$r_1\overline{\mathcal O}$, the same strict inequality holds for $t$ sufficiently close to $t_0$. Therefore
$\mathcal R^\eta(t)\geqslant r_1>r$ for all such $t$.
Thus $\{t:\mathcal R^\eta(t)>r\}$ is open, and
$\mathcal R^\eta$ is lower semicontinuous.

We claim that it is enough to prove
\begin{equation}\label{proof-GeneralW-1}
\liminf_{t\to+\infty}
\frac{\mathcal R^\eta(t)}{t}
\geqslant1
\text{ for every }\eta\in(0,1).
\end{equation}
Indeed, let $C\subset\operatorname{int}\mathcal S$ be compact. Since
$\mathcal S=\mathbf p+\mathcal O$
and $0\in\operatorname{int}\mathcal O$, there exists $\theta\in(0,1)$ such that
$C\subset\mathbf p+\theta\mathcal O$.
If \eqref{proof-GeneralW-1} holds, then
$\mathcal R^\eta(t)>\theta t $
for all sufficiently large $t$. Consequently,
$$
tC
\subset
t\mathbf p+\theta t\mathcal O
\subset
t\mathbf p+\mathcal R^\eta(t)\mathcal O,
$$
and hence
$\liminf\limits_{t\to+\infty}\inf_{x\in C}u(t,tx)\geqslant\eta$.
Since $\eta<1$ is arbitrary and $u\leqslant1$, the first part of \eqref{definitionw} follows.

We now prove \eqref{proof-GeneralW-1} by contradiction. Suppose that it fails. Since $\mathcal R^\eta(t)$ is nonincreasing with respect to $\eta$, we may choose $\eta_0\in(\alpha,1)$ and $k\in(0,1)$ such that
\begin{equation}\label{R-contradiction}
\liminf_{t\to+\infty}
\frac{\mathcal R^{\eta_0}(t)}{t}<k.
\end{equation}
For simplicity, write $\mathcal R(t):=\mathcal R^{\eta_0}(t)$.
Then
$\liminf_{t\to+\infty} \left(\mathcal R(t)-kt\right)=-\infty$.
For $n\in\mathbb N$, define
\[
t_n:= \inf\left\{t\geqslant0: \mathcal R(t)-kt\leqslant-n\right\}.
\]
The lower semicontinuity of $\mathcal R$ implies that the above infimum is attained. Thus
$\mathcal R(t_n)-kt_n\leqslant-n$,
while
$
\mathcal R(t)-kt>-n$ for $0\leqslant t<t_n$.
In particular,
$t_n\longrightarrow+\infty$,
and, for every $0\leqslant t<t_n$,
\begin{equation}\label{proof-GeneralW-2}
\mathcal R(t_n)-k(t_n-t)<\mathcal R(t) \text{ and }\mathcal R(t_n)\to+\infty \text{ as }n\to+\infty.
\end{equation}

By continuity and the maximality of $\mathcal R(t_n)$, there
exists $z_n\in\partial\mathcal O$ such that, with
$x_n:=t_n\mathbf p+\mathcal R(t_n)z_n$,
one has
$u(t_n,x_n)=\eta_0$. Define
\[
u_n(t,x):=u(t+t_n,x+x_n),\ t>-t_n,\ x\in\mathbb R^N.
\]
Since $t_n\to+\infty$, every compact subset of
$\mathbb R\times\mathbb R^N$ is contained in the domain of
$u_n$ for all sufficiently large $n$. We define the limit
\begin{equation}\label{definition-ustar}
u^*(t,x):=
\liminf_{\substack{n\to\infty\\(s,y)\to(t,x)}}
u_n(s,y),\
(t,x)\in\mathbb R_-\times\mathbb R^N.
\end{equation}
Then $0\leqslant u^*\leqslant1$ and $u^*$ is lower
semicontinuous.  Moreover, $u^*$ is a viscosity supersolution of
$$
v_t=\mathcal J\star v-v+f(v)
$$
in $\mathbb R_-\times\mathbb R^N$.
 For completeness,
if $\varphi$ touches $u^*$ from below at $(t_0,x_0)$, one can choose contact points $(t_n',x_n')\to(t_0,x_0)$ for $u_n$ such that $u_n(t_n',x_n')\to u^*(t_0,x_0).$
The viscosity inequality for $u_n$, together with
$\mathcal J\geqslant0$ and Fatou's lemma,
$$
(\mathcal J\star u^*)(t_0,x_0)
\leqslant
\liminf_{n\to\infty}
(\mathcal J\star u_n)(t_n',x_n'),
$$
yields the viscosity supersolution inequality for $u^*$. Since $u_n(0,0)=\eta_0$ for every $n$, the definition of
$u^*$ gives
\begin{equation}\label{proof-GeneralW-3}
u^*(0,0)\leqslant \eta_0.
\end{equation}

We now use \eqref{proof-GeneralW-2} to describe the region lying behind the contact point at earlier times. Fix $T\geqslant0$. For all sufficiently large $n$, $t_n>T$, and \eqref{proof-GeneralW-2}, applied with $t=t_n-T$, gives
$\mathcal R(t_n)-kT<\mathcal R(t_n-T)$.
Since the sets $r\mathcal O$ are increasing with respect to $r$, it follows from the definition of $\mathcal R(t_n-T)$ that
$$
u(t_n-T,y)>\eta_0\text{ for every }
y\in (t_n-T)\mathbf p+(\mathcal R(t_n)-kT)\mathcal O.
$$
After translating by $x_n=t_n\mathbf p+\mathcal R(t_n)z_n$, we obtain
\begin{equation}\label{proof-GeneralW-4}
u_n(-T,x)>\eta_0
\text{ for every }x\in A_{n,T}
:=
-T\mathbf p+(\mathcal R(t_n)-kT)\mathcal O-\mathcal R(t_n)z_n.
\end{equation}

We next identify the limiting geometry of the sets $A_{n,T}$. Since $\partial\mathcal O$ is compact, up to a subsequence,
$z_n\to z\in\partial\mathcal O$.
Let $\rho>0$ be a uniform interior ball radius for $\mathcal O$. For each $n$, choose an exterior unit normal $\nu_n$ to $\mathcal O$ at $z_n$ such that
$$
B_\rho(y_n)\subset\mathcal O,\
y_n:=z_n-\rho\nu_n,\
z_n\in\partial B_\rho(y_n)\cap\partial\mathcal O.
$$
Up to a further subsequence,
$\nu_n\to\nu\in\mathbb S^{N-1}$,
where $\nu$ is an exterior unit normal to $\mathcal O$ at $z$.

For every fixed $T\geqslant0$, define
\begin{equation}\label{proof-GeneralW-5}
\mathcal H_T
:=\left\{
x\in\mathbb R^N: x\cdot\nu<
-\big(\mathbf p\cdot\nu+kz\cdot\nu\big)T
\right\}.
\end{equation}
We claim that the sets $A_{n,T}$ eventually contain every point of $\mathcal H_T$, namely,
\begin{equation}\label{halfspace-exhaustion}
\mathcal H_T
\subset\bigcup_{M\in\mathbb N}
\bigcap_{n\geqslant M}A_{n,T}.
\end{equation}
Indeed, fix $x\in\mathcal H_T$. Since $\mathcal R(t_n)\to+\infty$,
$\mathcal R(t_n)-kT>0$ for all sufficiently large $n$.  It is therefore enough to prove that
$\dfrac{x+T\mathbf p+\mathcal R(t_n)z_n}{\mathcal R(t_n)-kT}\in\mathcal O.$
Using the interior tangent ball at $z_n$, we compute
\[
\begin{aligned}
\frac{x+T\mathbf p+\mathcal R(t_n)z_n}{\mathcal R(t_n)-kT}-y_n
=
\rho\nu_n+
\frac{x+T\mathbf p+kTz_n}{\mathcal R(t_n)-kT}.
\end{aligned}
\]
Set
$
\xi_n
:=
\dfrac{x+T\mathbf p+kTz_n}{\mathcal R(t_n)-kT}.
$
Then $\xi_n\to0$ and
\[
|\rho\nu_n+\xi_n|^2-\rho^2= \frac{
2\rho\nu_n\cdot(x+T\mathbf p+kTz_n)
}{\mathcal R(t_n)-kT}+
\frac{|x+T\mathbf p+kTz_n|^2}{(\mathcal R(t_n)-kT)^2}.
\]
Moreover,
\[
\nu_n\cdot(x+T\mathbf p+kTz_n)\longrightarrow
x\cdot\nu+ T\mathbf p\cdot\nu+ kTz\cdot\nu<0,
\]
where the last inequality follows from $x\in\mathcal H_T$. Hence
$|\rho\nu_n+\xi_n|<\rho$
for all sufficiently large $n$, and consequently
\[
\frac{x+T\mathbf p+\mathcal R(t_n)z_n}{\mathcal R(t_n)-kT}
\in B_\rho(y_n)\subset\mathcal O.
\]
Thus $x\in A_{n,T}$ for all sufficiently large $n$, proving \eqref{halfspace-exhaustion}.

Set
$a:=\mathbf p\cdot\nu+kz\cdot\nu.$
We claim that
\begin{equation}\label{proof-GeneralW-6}
u^*(t,x)\geqslant\eta_0
\text{ for }t\leqslant0,\ x\cdot\nu<a t.
\end{equation}
Indeed, fix $(t,x)$ such that $t<0$ and $x\cdot\nu<a t$.
By the strict inequality, there exist $\delta>0$ and a
neighborhood $\mathcal N$ of $(t,x)$ such that
$s<0$, $
y\cdot\nu-a s\leqslant-2\delta$
for all $(s,y)\in\mathcal N$.
For $(s,y)\in\mathcal N$, define
$$
q_n(s,y)
:=
\frac{y-s\mathbf p+\mathcal R(t_n)z_n}{\mathcal R(t_n)+ks}.
$$
For all sufficiently large $n$, $\mathcal R(t_n)+ks>0$ uniformly on
$\mathcal N$. Moreover,
\[
q_n(s,y)-y_n
=\rho\nu_n+\frac{y-s\mathbf p-ksz_n}{\mathcal R(t_n)+ks}.
\]
Hence\[
|q_n(s,y)-y_n|^2-\rho^2
=
\frac{2\rho\,\nu_n\cdot
(y-s\mathbf p-ksz_n)}{\mathcal R(t_n)+ks}
+
\frac{|y-s\mathbf p-ksz_n|^2}{(\mathcal R(t_n)+ks)^2}.
\]
Since $\nu_n\to\nu$ and $z_n\to z$, uniformly for
$(s,y)\in\mathcal N$,
$$
\nu_n\cdot(y-s\mathbf p-ksz_n)
\longrightarrow
y\cdot\nu-a s
\leqslant-2\delta.
$$
Thus, for all sufficiently large $n$,
$
\nu_n\cdot(y-s\mathbf p-ksz_n)\leqslant-\delta
$ for all $(s,y)\in\mathcal N$.
Since $\mathcal N$ is bounded, there exists $C>0$ such that
$|y-s\mathbf p-ksz_n|\leqslant C$ on $\mathcal N$. Therefore,
for all sufficiently large $n$,
$$
|q_n(s,y)-y_n|^2-\rho^2
\leqslant
-\frac{2\rho\delta}{\mathcal R(t_n)+ks}
+
\frac{C^2}{(\mathcal R(t_n)+ks)^2}
<0
$$
uniformly for $(s,y)\in\mathcal N$.
Consequently,
$
q_n(s,y)\in B_\rho(y_n)\subset\mathcal O
$
for every $(s,y)\in\mathcal N$ and all sufficiently large $n$. Equivalently,
$
y\in
s\mathbf p+(\mathcal R(t_n)+ks)\mathcal O-\mathcal R(t_n)z_n$. Hence by \eqref{proof-GeneralW-4},
$
u_n(s,y)>\eta_0$
for all $(s,y)\in\mathcal N $
and all sufficiently large $n$.
The definition of the lower  limit therefore yields
\eqref{proof-GeneralW-6}. Equivalently, for every $T>0$,
\[
u^*(-T,x)\geqslant\eta_0
\text{ for every }T\geqslant0,\ x\in\mathcal H_T.
\]
    In particular, for $t \leqslant 0$ and $x \in \mathcal{H}_{-t} = \{ x : x \cdot \nu < (\mathbf p\cdot\nu+k z\cdot\nu )t \}$, we have $u^*(t, x) \geqslant \eta_0$.
Then \eqref{proof-GeneralW-6} can equivalently be written as
\begin{equation}\label{behind-interface}
u^*(t,x)\geqslant \eta_0
\text{ for }t\leqslant 0
\text{ and }x\cdot\nu<(\mathbf p\cdot\nu+kz\cdot\nu)t.
\end{equation}
The number $\mathbf p\cdot\nu+kz\cdot\nu$ is strictly smaller than the critical wave speed in the direction $\nu$. Indeed, since $0\in\operatorname{int}\mathcal O$ and $\nu$ is an exterior normal to $\mathcal O$ at $z$,
$z\cdot\nu>0.$
Hence, using $k<1$,
$$
\begin{aligned}
\mathbf p\cdot\nu+kz\cdot\nu<
\mathbf p\cdot\nu+z\cdot\nu\
\leqslant c_\nu^*.
\end{aligned}
$$
Therefore,
$\mathbf p\cdot\nu+kz\cdot\nu<c_\nu^*$.
We now pass to the frame moving with $\mathbf p+kz$. Set
$\zeta:=\mathbf p+kz$
and define
\begin{equation}\label{definition-ubar}
\overline u(t,x)
:=
u^*(t,x+\zeta t),\ t\leqslant 0.
\end{equation}
By \eqref{proof-GeneralW-3} and  \eqref{behind-interface},
\[
\overline u(0,0)\leqslant \eta_0,
\text{ and }\overline u(t,x)\geqslant \eta_0
\text{ for all }t\leqslant 0\text{ and } x\cdot\nu<0
.
\]
Moreover, $\overline u$ is  a viscosity supersolution of
\begin{equation}\label{equationv2}
\overline u_t=\mathcal J\star\overline u-\overline u
+\zeta\cdot\nabla\overline u
+f(\overline u),\
t<0,\ x\in\mathbb R^N.
\end{equation}
Since $\eta_0>\alpha$, the half-space estimate above and Lemma \ref{mathcalH} yield
\begin{equation}\label{proof-GeneralW-7}
\liminf_{x\cdot\nu\to-\infty}
\overline u(t,x)\geqslant 1
\text{ uniformly in }t\leqslant0.
\end{equation}

It remains to compare $\overline u$ with a traveling front in the direction $\nu$. Fix $\nu\in\mathbb S^{N-1}$ and extend $f$ to $(-\infty,1]$ by setting $f(s)=0$ for $s\leqslant0$. For every
$\varepsilon>0$, the restriction of this extended nonlinearity to
$[-\varepsilon,1]$ is of combustion type. Moreover, the projected
kernel $J_\nu$ is a probability density with finite first moment
and, by {\rm (J)}, its support contains a neighborhood of the
origin. Hence the traveling-wave theory in \cite{C2007} yields a
unique speed $c_{\varepsilon,\nu}\in\mathbb R$ for which the
projected equation admits a nonincreasing traveling front
$\phi_\nu^\varepsilon$ connecting $1$ to $-\varepsilon$.
We claim that
\begin{equation}\label{combustion-limit}
c_{\varepsilon,\nu}\nearrow c_\nu^*
\text{ as }\varepsilon\searrow0.
\end{equation}
The argument is adapted from \cite[Proposition~2.6]{R2017}.
We first show that
$c_{\varepsilon,\nu}\leqslant c_\nu^*$.
Suppose, by contradiction, that
$c_{\varepsilon,\nu}>c_\nu^*$.
Choose $
c_\nu^*<c<c_{\varepsilon,\nu}$ with $c\neq0$,
and let $\phi_\nu^c$ be a KPP front in the direction $\nu$
with speed $c$. In the frame moving with velocity $c\nu$,
this KPP front is stationary, whereas the combustion front moves with the positive relative speed $\gamma:=c_{\varepsilon,\nu}-c>0$.
More precisely, up to translations of the profiles, set
\[
\overline v(t,x):=\phi_\nu^c(x\cdot\nu),\
\underline v(t,x):=\phi_\nu^\varepsilon (x\cdot\nu-\gamma t-a),\ t\leqslant0,
\]
where $a\in\mathbb R$ is chosen so that
$\underline v(0,0)>\overline v(0,0)$.
Both functions solve the equation in the moving frame with drift
$c\nu$. Furthermore, $\overline v\geqslant0$ and
$\overline v(t,x)\to1$ as $x\cdot\nu\to-\infty$ uniformly in
$t\leqslant0$, while, since
$\phi_\nu^\varepsilon(+\infty)=-\varepsilon<0$, there exists
$L\in\mathbb R$ such that
\[\underline v(t,x)\leqslant0\text{ whenever }
x\cdot\nu\geqslant\gamma t+L.
\]
Lemma~\ref{CP} therefore yields
$\underline v\leqslant\overline v$ on
$\mathbb R_-\times\mathbb R^N$, contradicting the choice of the translation. Hence $c_{\varepsilon,\nu}\leqslant c_\nu^*$.

We next observe that $c_{\varepsilon,\nu}$ is nonincreasing with respect to $\varepsilon$. Indeed, let
$0<\varepsilon_1<\varepsilon_2$. The proof of
Lemma~\ref{CP} is unchanged if the reference level $0$ in
\eqref{CP-1}--\eqref{CP-2} is replaced by any
$a\in[-\varepsilon_2,0]$. Arguing as above and taking
$a=-\varepsilon_1$ for the fronts connecting
$1$ to $-\varepsilon_1$ and $1$ to $-\varepsilon_2$, respectively,
we obtain
$c_{\varepsilon_2,\nu}\leqslant c_{\varepsilon_1,\nu}$.
Thus, as $\varepsilon\to0$,
$c_{\varepsilon,\nu}$ is nondecreasing. Moreover, for any fixed $\varepsilon_0>0$, $
c_{\varepsilon_0,\nu}
\leqslant c_{\varepsilon,\nu}
\leqslant c_\nu^*$
for $0<\varepsilon<\varepsilon_0$.
Hence there exists $c_0\in\mathbb R$ with
$c_0\leqslant c_\nu^*$ such that
$c_{\varepsilon,\nu}\to  c_0$ as $\varepsilon\to0$.

It remains to identify $c_0$. Fix $\eta_1\in(0,1)$ and translate
each profile so that
\[
\phi_\nu^\varepsilon(z)\geqslant\eta_1
\quad\text{for }z<0,
\qquad
\phi_\nu^\varepsilon(z)\leqslant\eta_1
\quad\text{for }z>0.
\]
Since the profiles are nonincreasing and
$-\varepsilon\leqslant\phi_\nu^\varepsilon\leqslant1$,
Helly's theorem yields, along any sequence
$\varepsilon_n\to0$, a subsequence, still denoted by
$\varepsilon_n$, and a nonincreasing function
$\phi_0:\mathbb R\to[0,1]$ such that
$\phi_\nu^{\varepsilon_n}\to\phi_0$ almost everywhere and in
$L^1_{\rm loc}(\mathbb R)$. A standard truncation argument,
using $J_\nu\in L^1(\mathbb R)$ and the uniform boundedness of the profiles, allows us to pass to the limit in the convolution term. Passing to the limit in the profile equation in the distributional sense, we obtain
\[
J_\nu*\phi_0-\phi_0+c_0\phi_0'+f(\phi_0)=0
\text{ in }\mathbb R.
\]
Since $\phi_0$ is nonincreasing, the limits
$\ell_\pm:=\phi_0(\pm\infty)$ exist. Integrating the above
equation over intervals translated to $\pm\infty$ and using
dominated convergence yields $f(\ell_\pm)=0$.
Since $f>0$ on $(0,1)$ and the normalization implies
$\ell_-\geqslant\eta_1$ and
$\ell_+\leqslant\eta_1$, we necessarily have
$\phi_0(-\infty)=1$ and
$\phi_0(+\infty)=0$.
Thus $\phi_0$ is a traveling front for
\eqref{traveling-wave} in the direction $\nu$ with speed $c_0$.
By the definition of the minimal KPP speed,
$c_0\geqslant c_\nu^*$. Combined with
$c_0\leqslant c_\nu^*$, this yields
$c_0=c_\nu^*$ and proves \eqref{combustion-limit}.

Since
 $\mathbf p\cdot\nu+kz\cdot\nu<c_\nu^*$
and $c_{\varepsilon,\nu}\to c_\nu^*$ as
$\varepsilon\to0$, we may choose $\varepsilon>0$
sufficiently small such that $\mathbf p\cdot\nu+kz\cdot\nu<c_{\varepsilon,\nu}$.
After a further translation of the profile, which does not change its speed, we may assume that
 $\phi_\nu^\varepsilon(0)>\eta_0$.
Set
\[
v(t,x)
:=
\phi_\nu^\varepsilon
\bigl(x\cdot\nu-c_{\varepsilon,\nu}t\bigr).
\]
Then $v\leqslant1$, $v(0,0)>\eta_0$, and there exists
$L\in\mathbb R$ such that
\begin{equation}\label{proof-GeneralW-8}
v(t,x)\leqslant0 \text{ for }t\leqslant0,\
x\cdot\nu\geqslant c_{\varepsilon,\nu}t+L.
\end{equation}
Consider this subsolution in the moving frame:
\[
\underline u(t,x):=v(t,x+\zeta t).
\]
The function $\underline u$ is a viscosity subsolution of the same shifted equation~\eqref{equationv2} as $\overline u$.
Condition~\eqref{proof-GeneralW-8} transforms into
\[
\underline u(t,x)\leqslant0 \text{ for }t\leqslant0,\
x\cdot\nu\geqslant(c_{\varepsilon,\nu}-(\mathbf p\cdot\nu+kz\cdot\nu))t+L.
\]
Since
$\gamma:=c_{\varepsilon,\nu}-(\mathbf p\cdot\nu+kz\cdot\nu)>0$,
this is precisely the decay condition required in
Lemma~\ref{CP}.

We are now in a position to apply Lemma~\ref{CP} to
$\underline u$ and $\overline u$ with direction $e=\nu$.
It follows that
\[
\underline u(t,x)\leqslant\overline u(t,x)
\qquad
\text{for all }t\leqslant0,\ x\in\mathbb R^N.
\]
In particular, at the origin,
\[
\underline u(0,0)
=
v(0,0)
\leqslant
\overline u(0,0)
\leqslant\eta_0,
\]
which contradicts the normalization
$v(0,0)>\eta_0$. This contradiction proves
\eqref{proof-GeneralW-1}. Consequently, $\mathcal S$ satisfies
the first part of \eqref{definitionw} under assumption
\eqref{additional}.

We now remove the additional assumption \eqref{additional}.
Let $C\subset\operatorname{int}\mathcal S$ be an arbitrary
compact set. Recall that
$\mathcal O=\mathcal S-\mathbf p$ is a bounded open convex set
containing the origin. Since $C-\mathbf p$ is a compact subset
of $\mathcal O$, there exists $\theta\in(0,1)$ such that
$C-\mathbf p\subset\theta\mathcal O$.
Moreover, since
$\theta\overline{\mathcal O}\Subset\mathcal O$, we may choose
$\delta>0$ sufficiently small so that
$\theta\overline{\mathcal O}+\overline{B_\delta}
\subset\mathcal O$.
Set
$\mathcal O_{\theta,\delta}:=
\theta\overline{\mathcal O}+B_\delta$ and
$\mathcal S_{\theta,\delta}:=\mathbf p+\mathcal O_{\theta,\delta}$.
Then $\mathcal S_{\theta,\delta}$ is a bounded open convex set,
\[
C\subset\mathcal S_{\theta,\delta}
\subset\mathcal S,\
\mathbf p\in\operatorname{int}\mathcal S_{\theta,\delta},
\]
and $\mathcal S_{\theta,\delta}$ satisfies the uniform interior
ball condition. Indeed, $\mathcal O_{\theta,\delta}$ is the
$\delta$-parallel neighborhood of the compact convex set
$\theta\overline{\mathcal O}$. We may now repeat the preceding argument with $\mathcal O_{\theta,\delta}$ in place of $\mathcal O$. Thus the crucial strict inequality used in the preceding traveling-wave comparison remains valid, and the same argument yields
\[
\lim_{t\to+\infty}\inf_{x\in C}u(t,tx)=1.
\]
Since $C\subset\operatorname{int}\mathcal S$ was arbitrary, the first part of \eqref{definitionw} follows without assumption \eqref{additional}.

It remains to prove the second part of \eqref{definitionw}.
By compactness, it is enough to consider a compact convex set
$C\subset \mathbb R^N\setminus\overline{\mathcal S}$ sufficiently small so that,
by the definition of $\mathcal S$, there exist $e\in\mathbb S^{N-1}$ and
$c_1>c>c_e^*$ such that
$C\subset\{x\in\mathbb R^N:x\cdot e\geqslant c_1\}.$
Indeed, an arbitrary compact subset of
$\mathbb R^N\setminus\overline{\mathcal S}$ can be covered by finitely many such compact convex sets.

Since
$
c_e^*=\inf_{\lambda>0}G_e(\lambda)<c,$
we may choose $\lambda_c>0$ such that
$G_e(\lambda_c)<c$.
Define $v(t,x)=e^{-\lambda_c(x\cdot e-ct-R)}$.
For $R>0$ sufficiently large, the compact support of $u_0$ ensures that $v(0,x)\geqslant u_0(x)$ for all $x\in\mathbb R^N$. Moreover, using the KPP property $f(s)\leqslant f'(0)s$, one readily checks that
$v$ is a supersolution of \eqref{equation}. Hence, by the comparison principle, $u(t,x)\leqslant v(t,x).$
Therefore, for every $x\in C$,
\[
u(t,tx)
\leqslant
e^{-\lambda_c(tx\cdot e-ct-R)}
\leqslant
e^{-\lambda_c((c_1-c)t-R)}.
\]
Since $c_1>c$, it follows that
\[
\lim_{t\to+\infty}\sup_{x\in C}u(t,tx)=0.
\]
By the finite covering argument above, the same conclusion holds for every compact
$C\subset\mathbb R^N\setminus\overline{\mathcal S}$. Thus $\mathcal S$ satisfies
the second part of \eqref{definitionw}.
\end{proof}

\subsection{Unbounded Initial Support} Intuitively, the set $\mathcal W$ defined in  \eqref{th2-2} separates the asymptotic  regions where the solution $u$ converges to 1 and to 0 as $t\to+\infty$ for more general unbounded initial supports.
 \begin{lemma}\label{Udelta+ts}  Let  $u(t,x)$ be the solution of equation~\eqref{equation} with  $u_0(x)=\mathbbm{1}_U$, where $U \subset \mathbb{R}^N$ satisfies $U_\delta \neq \varnothing$ for some $\delta>0$. Let $\mathcal S''$ be any closed bounded subset of the interior of $\mathcal S$ defined in \eqref{mathcals}.  Then,
\[
\inf _{x \in U_\delta+t\mathcal S''} u(t, x) \rightarrow 1  \text { as } t \rightarrow+\infty .
\]
\end{lemma}
\begin{proof} Let $v$ be the solution to equation~\eqref{equation} with  the initial datum $v_0=\mathbbm{1}_{B_\delta}$.  Fix  $\lambda\in(0,1)$. It follows from~\eqref{S} that  there exists $t_\lambda>0$ such that
$$
\text{ for all } t \geqslant t_\lambda  \text{ and } x \in t\mathcal S'',\ v(t, x)>\lambda.
$$
Now, for any $x_0 \in U_\delta$, we have  $u_0 \geqslant v_0\left(\cdot-x_0\right)$ on $\mathbb{R}^N$. By the comparison principle,
$$
 \text{ for all }t \geqslant t_\lambda \text{ and } x \in x_0+t\mathcal S'',  u(t, x) \geqslant  v\left(t, x-x_0\right)>\lambda.
$$
Since this holds for any $x_0 \in U_\delta$ and any $\lambda\in(0,1)$, the result is established.
\end{proof}

\begin{lemma}\label{w>w(e)}Assume that  the hypotheses of Theorem \ref{th2} hold. Let $\mathcal S$ be given by \eqref{se} and $\mathcal W$ be given by \eqref{th2-1} and \eqref{th2-2}. If $w(e)<+\infty$ for some $e \in \mathbb S^{N-1}$,  then for any $w> w(e)$ and
$y_0:=\mathbf p+we$,    there exists an open neighborhood $\mathcal N_{y_0}$ of $y_0$ such that
\begin{equation}\label{mathcalC}
\lim_{t\to+\infty}
\sup_{y\in\mathcal N_{y_0}}
u(t,ty)=0.
\end{equation}
\end{lemma}
 \begin{proof}
The proof is based on a finite family of planar traveling waves whose propagation directions separate the point $y_0$ from the unbounded directions of the initial support.

\textit{Step 1. Directional separation and finite covering.} Assume first that $\mathcal U(U)\neq\varnothing$. Fix
$\xi\in\mathcal U(U)$. By the variational characterization
\eqref{th1-2},
\[
w>w(e) \geqslant
\inf_{
\eta\cdot e>0,
\eta\cdot\xi\leqslant0}
\frac{c_\eta^*-\mathbf p\cdot\eta}
{\eta\cdot e},\ \eta\in\mathbb S^{N-1}.
\]
Hence there exists
$\eta_\xi\in\mathbb S^{N-1}$ such that
$
\eta_\xi\cdot e>0$,
$\eta_\xi\cdot\xi\leqslant0,$
and
$
w>\frac{c_{\eta_\xi}^*-\mathbf p\cdot\eta_\xi}{\eta_\xi\cdot e}.
$
Equivalently,
\begin{equation}\label{strict-y-eta}
y_0\cdot\eta_\xi
=\mathbf p\cdot\eta_\xi
+we\cdot\eta_\xi >c_{\eta_\xi}^*.
\end{equation}
If $\eta_\xi\cdot\xi=0$, we may perturb $\eta_\xi$ slightly
towards $-\xi$ and renormalize. By the continuity of
$\eta\mapsto c_\eta^*$, for a sufficiently small perturbation
we retain \eqref{strict-y-eta} and
$\eta_\xi\cdot e>0$, while obtaining
\begin{equation}\label{strict-xi-eta}
\eta_\xi\cdot\xi<0.
\end{equation}
We henceforth assume that \eqref{strict-xi-eta} holds. Since the interval
$\bigl(c_{\eta_\xi}^*,\,y_0\cdot\eta_\xi\bigr)$
is nonempty, we can choose $
c_\xi \in
(c_{\eta_\xi}^*,\,y_0\cdot\eta_\xi)
\setminus\{0\}$.
By \eqref{strict-xi-eta}, there exists an open neighborhood
$V_\xi$ of $\xi$ in $\mathbb S^{N-1}$ such that $
\zeta\cdot\eta_\xi<0$ for every $\zeta\in V_\xi$.

Under condition \eqref{SN-1}, we have
$\mathcal U(U)=\mathcal U(U_\delta)$
 and $\mathcal U(U)$ is compact. The family $\{V_\xi\}_{\xi\in\mathcal U(U)}$ therefore admits
a finite subcover. Hence there exist
$\xi_1,\ldots,\xi_m\in\mathcal U(U)$
such that
$\mathcal U(U) \subset\bigcup_{i=1}^m V_{\xi_i}$.
For simplicity, set
$\eta_i:=\eta_{\xi_i}$,
$c_i:=c_{\xi_i}$  and $y_0\cdot\eta_i-c_i>0$ for $i=1,\cdots,m$.

We next show that the same finite family controls the
unbounded part of $U$. Set
$\mathcal V:=\bigcup_{i=1}^mV_{\xi_i}$.
We claim that there exists $R_0>0$ such that
\begin{equation}\label{farU}
\hat x
\in\mathcal V
\text{ for every }x\in U\text{ with }|x|>R_0.
\end{equation}
Suppose otherwise. Then there exist
$x_n\in U$ such that
$|x_n|\to+\infty$ and
$\hat  x_n \notin\mathcal V$.
Passing to a subsequence, we may assume that
$\hat  x_n\to\xi \in\mathbb S^{N-1}\setminus\mathcal V$.
Since $x_n\in U$,
\[
\frac{\operatorname{dist}(|x_n|\xi,U)
}{|x_n|} \leqslant
|\xi-\hat x_n| \to 0.
\]
Hence $\xi\notin\mathcal B(U)$. By \eqref{SN-1},
$
\xi\in\mathcal U(U_\delta)
=\mathcal U(U)\subset\mathcal V,$
which is a contradiction. Thus \eqref{farU} holds.
Consequently, for every $x\in U$ with $|x|>R_0$, there exists
$i\in\{1,\ldots,m\}$ such that
$ \hat x\in V_{\xi_i}$ and $x\cdot\eta_i<0$.

We now consider the case $\mathcal U(U)=\varnothing$. Then
condition \eqref{SN-1} implies that $U$ is bounded.  In this case,
$\mathcal W=\mathcal S$, $w(e)=\sigma(e)$.
Since $w>\sigma(e)$, by the variational formula for
$\sigma(e)$ there exists $\eta_1\in\mathbb S^{N-1}$ such that $\eta_1\cdot e>0$ and  $y_0\cdot\eta_1>c_{\eta_1}^*$.
Choose
$c_1
\in \bigl(c_{\eta_1}^*,\,y_0\cdot\eta_1\bigr)
\setminus\{0\}$.
We then take $m=1$ and choose $R_0>0$ so large that
$U\subset B_{R_0}$.
Thus, in either case, we have obtained finitely many pairs
$(\eta_i,c_i)$, $i=1,\ldots,m$, such that
\begin{equation}\label{finite-speed-gap}
c_i>c_{\eta_i}^*,\
c_i\neq0,\ y_0\cdot\eta_i>c_i,
\end{equation}
and, whenever $U$ is unbounded, every
$x\in U\setminus B_{R_0}$ satisfies
$x\cdot\eta_i<0$ for at least one index $i$.

Set
$\delta:= \frac{1}{2}\min_{1\leqslant i\leqslant m}\bigl(y_0\cdot\eta_i-c_i\bigr)>0$.
Since the family is finite, there exists an open neighborhood
$\mathcal N_{y_0}$ of $y_0$ such that
\begin{equation}\label{uniform-speed-gap}
y\cdot\eta_i-c_i \geqslant\delta
\text{ for every } y\in\mathcal N_{y_0}, i=1,\cdots,m.
\end{equation}

\textit{Step 2. Construction of a family of supersolutions.} For each $i=1,\ldots,m$, since
$c_i>c_{\eta_i}^*$ and
$c_i\neq0$,
Proposition~\ref{Prop-1} yields a decreasing $C^1$ traveling wave profile $\phi_i$ with speed $c_i$ in the direction $\eta_i$. We normalize it by
$\phi_i(0)=\frac{1}{2}$.
 Fix $\varepsilon\in(0,1)$. Since
$\phi_i(+\infty)=0$, we can choose $L>0$ so large that
$2\sum_{i=1}^m\phi_i(L)<\varepsilon$.
For $T>0$ and $y\in\mathcal N_{y_0}$, define
\begin{equation}\label{supersolution-vTy}
\overline v^{T,y}(t,x):=
2\sum_{i=1}^m
\phi_i\left((x-Ty)\cdot\eta_i-c_i(t-T)+L
\right),\ 0\leqslant t\leqslant T.
\end{equation} Direct computation yields$$\begin{aligned}\partial_t \bar v^{T,y}-\mathcal{J}\star \bar v^{T,y}+\bar v^{T,y}-f\left(\bar v^{T,y}\right)\geqslant &2 \sum_{i=1}^m \left[ -c_i\phi_i'-\mathcal J\star\phi_i+\phi_i \right]-2 \sum_{i=1}^mf(\phi_i)
= 0,
 \end{aligned}$$
 Here we use the extension  $f(s)=0$ for $s\geqslant1$ and $f(2\sum_{i=1}^m\phi_i)\leqslant 2\sum_{i=1}^mf(\phi_i)$. Thus, $\overline{v}^{T,y}$ is a positive supersolution of \eqref{equation} on $[0, T] \times \mathbb{R}^N$.

\textit{Step 3. Uniform comparison with the initial datum.}
We now show that, for all sufficiently large $T$,
$u_0(x) \leqslant\overline v^{T,y}(0,x)$
 for every $x\in\mathbb R^N$ and every $y\in\mathcal N_{y_0}$. If $x\notin U$, then $u_0(x)=0$, whereas
$\overline v^{T,y}(0,x)\geqslant0$, so there is nothing to
prove. Let $x\in U$ and $|x|\leqslant R_0$. For every
$i=1,\cdots,m$, by \eqref{uniform-speed-gap},
\[
(x-Ty)\cdot\eta_i+c_iT+L =x\cdot\eta_i -T(y\cdot\eta_i-c_i)+L \leqslant R_0-\delta T+L.\]
If $x\in U$ and $|x|>R_0$, then, in the unbounded case, there exists an index $i$ such that $x\cdot\eta_i<0$, and hence
\[
(x-Ty)\cdot\eta_i+c_iT+L=x\cdot\eta_i -T(y\cdot\eta_i-c_i)+L\leqslant-\delta T+L.\]
In the case $\mathcal U(U)=\varnothing$, the latter situation
does not occur because $U\subset B_{R_0}$. Choose $T_\varepsilon>0$ sufficiently large so that
$R_0-\delta T_\varepsilon+L\leqslant0$.
Then, for every $T\geqslant T_\varepsilon$, every
$y\in\mathcal N_{y_0}$ and every $x\in U$, there exists at
least one index $i$ for which the argument of $\phi_i$ at $(0,x)$ is nonpositive. Since
$\phi_i$ is decreasing and $\phi_i(0)=1/2$,
$\overline v^{T,y}(0,x)\geqslant2\phi_i(0)=1=u_0(x)$.

By the comparison principle,
$
u(t,x)\leqslant
\overline v^{T,y}(t,x)$
for
$0\leqslant t\leqslant T$, $x\in\mathbb R^N$, $y\in\mathcal N_{y_0}$ and
$T\geqslant T_\varepsilon$.
Evaluating at $t=T$ and $x=Ty$, \eqref{supersolution-vTy} gives
\[
u(T,Ty)\leqslant\overline v^{T,y}(T,Ty)=2\sum_{i=1}^m\phi_i(L)<\varepsilon.
\]
The constants $L$, $\delta$, $R_0$, and $T_\varepsilon$ are
independent of $y\in\mathcal N_{y_0}$. Hence
\[
\sup_{y\in\mathcal N_{y_0}}u(T,Ty)
<
\varepsilon
\text{ for every }T\geqslant T_\varepsilon.
\]
Since $\varepsilon\in(0,1)$ was arbitrary, this proves
\eqref{mathcalC}.
 \end{proof}

   We present the proofs of Theorems~\ref{th2}--\ref{th1} in the following. The logical order of the proofs is as follows: Theorem~\ref{th3}, Theorem~\ref{th2}, and finally Theorem~\ref{th1}.
 \begin{proof}[Proof of Theorem \ref{th3}]
It suffices to establish the two one-sided Hausdorff estimates:
\begin{equation}\label{proof-th3-1}
\sup_{x \in U+t\mathcal{S}} \operatorname{dist}\left(x, F_\lambda(t)\right) = o(t) \quad \text{as } t \to +\infty,
\end{equation}
and
\begin{equation}\label{proof-th3-2}
\sup_{x \in F_\lambda(t)} \operatorname{dist}\left(x, U+t\mathcal{S}\right) = o(t) \quad \text{as } t \to +\infty.
\end{equation} Here, both conditions $U_\delta \neq \varnothing$ and $d_{\mathcal{H}}(U, U_\delta) < +\infty$ are crucial.

First, since $U_\delta \neq \varnothing$ for some $\delta > 0$, the set $U$ contains at least one ball of radius $\delta$. For any fixed $\varepsilon\in(0,1)$, define the contracted closed  convex set
 $\mathcal{S}_\varepsilon :=\mathbf p+(1-\varepsilon)\overline{\mathcal O}$ and $  {\mathcal S_\varepsilon}\subseteq \mathcal S$. Denote $R_\mathcal S=\sup_{x\in\mathcal S}|x-\mathbf p|<+\infty$ and $d_{\mathcal H}(\mathcal S_\varepsilon,\mathcal S)\leqslant \varepsilon R_\mathcal S$.
 By Lemma \ref{Udelta+ts}, the existence of a nonempty positive interior $U_\delta$ ensures the standard spreading property from $U_\delta$:
 \[\lim_{t \to +\infty} \inf_{x \in U_\delta + t\mathcal{S}_\varepsilon} u(t, x) = 1.\]
 Thus, for any $\lambda \in (0, 1)$, there exists $T_\varepsilon > 0$ such that $U_\delta + t\mathcal{S}_\varepsilon \subset F_\lambda(t)$ for all $t > T_\varepsilon$.

  Second, by the finite Hausdorff distance assumption $d_{\mathcal{H}}(U, U_\delta) < +\infty$, there exists a uniform constant $M > d_{\mathcal{H}}(U, U_\delta) \in [0, +\infty)$ such that\begin{equation}\label{U-delta-bound}\text{for every } y \in U, \text{ there exists } y' \in U_\delta \text{ with } |y - y'| \leqslant M.\end{equation}
  Now, take any point $x \in U + t\mathcal{S}$ and decompose it as $x = y + ts$ with $y \in U$ and $s \in \mathcal{S}$. By \eqref{U-delta-bound}, choose $y' \in U_\delta$ satisfying $|y - y'| \leqslant M$, and choose $s' \in \mathcal{S}_\varepsilon$ satisfying $|s - s'| \leqslant  \varepsilon R_\mathcal S$. Setting $x' := y' + t s' \in U_\delta + t\mathcal{S}_\varepsilon \subset F_\lambda(t)$ for $t > T_\varepsilon$, we estimate
 \[\operatorname{dist}\left(x, F_\lambda(t)\right) \leqslant |x - x'| \leqslant |y - y'| + t|s - s'| \leqslant M + \varepsilon R_\mathcal S t.\]
 Dividing by $t$ and taking the supremum over $x \in U + t\mathcal{S}$ yields
 \[\frac{1}{t} \sup_{x \in U+t\mathcal{S}} \operatorname{dist}\left(x, F_\lambda(t)\right) \leqslant \frac{M}{t} +  \varepsilon R_\mathcal S \Longrightarrow \limsup_{t \to +\infty} \frac{\sup_{x \in U+t\mathcal{S}} \operatorname{dist}\left(x, F_\lambda(t)\right)}{t} \leqslant \varepsilon R_\mathcal S.\]Since $\varepsilon > 0$ is arbitrary, sending $\varepsilon \to 0$ establishes \eqref{proof-th3-1}.

We now prove \eqref{proof-th3-2}.
Let $v(t,x)$ be the solution of the linearized problem
\begin{equation}\label{linearized-problem}
\left\{
\begin{array}{lcl}
v_t=\mathcal J*v-v+f'(0)v,
&&t>0,\ x\in\mathbb R^N,\\
v(0,x)=\mathbbm 1_U(x), && x\in\mathbb R^N.
\end{array}
\right.
\end{equation}
Since $0\leqslant u\leqslant1$ and $0\leqslant f(u)\leqslant f'(0)u$, it follows  from  the comparison principle that
$0\leqslant u(t,x)\leqslant v(t,x)$ for all $t>0$ and  $x\in\mathbb R^N$.
Define the probability measure
\begin{equation}\label{Pt}
P_t
:=
e^{-t}
\sum_{k=0}^{+\infty}
\frac{t^k}{k!}\mathcal J^{(*k)},
\end{equation}
where $\mathcal J^{(*0)}=\delta_0$ and
$\mathcal J^{(*k)}$ denotes the $k$-fold convolution of
$\mathcal J$ for $k\geqslant1$. Then the solution of
\eqref{linearized-problem} is given by
\begin{equation}\label{linear-representation}
v(t,x)=e^{f'(0)t}\int_{\mathbb R^N}
\mathbbm 1_U(x-z)\,P_t(dz)=e^{f'(0)t}P_t(x-U).
\end{equation}
For any  $y\in\mathbb R^N$, set
$\Lambda(y):=\int_{\mathbb R^N}
\mathcal J(x)e^{y\cdot x}dx$ and  $H(y):=\Lambda(y)-1+f'(0)$,
and let $\mathcal D:=\{y\in\mathbb R^N:\Lambda(y)<+\infty\}$.
By assumption  (J), $\mathcal D$ contains a neighborhood
of the origin. Moreover, for every $y\in\mathcal D$, by Tonelli's theorem, we have
\begin{equation}\label{PTE}
\int_{\mathbb R^N}e^{y\cdot z}\,P_t(dz)
=
e^{-t}\sum_{k=0}^{+\infty}\frac{t^k}{k!}\int_{\mathbb R^N}e^{y\cdot z}\mathcal J^{(*k)}(dz)
=e^{-t}\sum_{k=0}^{+\infty}
\frac{t^k}{k!}\Lambda^k(y)=e^{t(\Lambda(y)-1)}.
\end{equation}

For every $\varepsilon>0$, there exist
$y_1,\cdots,y_m\in\mathcal D\setminus\{0\}$ and
$\gamma_\varepsilon>0$ such that
\begin{equation}\label{finite-separation}
E_\varepsilon:=\left\{x\in\mathbb R^N:\operatorname{dist}(x,\overline{\mathcal S})\geqslant\varepsilon\right\}
\subset\bigcup_{i=1}^m
\left\{x\in\mathbb R^N:
y_i\cdot x\geqslant H(y_i)+\gamma_\varepsilon
\right\}.
\end{equation}
To prove \eqref{finite-separation}, we  observe that
\begin{equation}\label{S-H-representation}
\overline{\mathcal S}
= \bigcap_{y\in\mathcal D\setminus\{0\}}
\left\{
x\in\mathbb R^N:
y\cdot x\leqslant H(y)
\right\}.
\end{equation}
Indeed, let $s\in\overline{\mathcal S}$ and write
$y=|y|\hat y$. By the definition of
$c_{\hat y}^*$,
$
c_{\hat y}^*=\inf_{\lambda>0}
\frac{H(\lambda \hat y)}{\lambda}
\leqslant\frac{H(y)}{|y|}.
$
Since $s\cdot {\hat y}\leqslant c_{\hat y}^*$, we obtain
$
y\cdot s= |y|(s\cdot \hat y)\leqslant
|y|c_{\hat y}^*\leqslant H(y).$
This proves one inclusion in \eqref{S-H-representation}.
Conversely, suppose that
$y\cdot x\leqslant H(y)$
for every $y\in\mathcal D\setminus\{0\}$.
Then, for every $e\in\mathbb S^{N-1}$ and every
$\rho>0$ such that $\rho e\in\mathcal D$,
$
x\cdot e\leqslant\frac{H(\rho e)}{\rho}.
$
Taking the infimum over $\rho$ gives
$x\cdot e\leqslant c_e^*$
for every $e\in\mathbb S^{N-1}$.
Hence $x\in\overline{\mathcal S}$ by the definition of
$\overline{\mathcal S}$. Thus \eqref{S-H-representation} holds.

Now fix $\varepsilon>0$. For every
$x_0\in E_\varepsilon$, we have
$x_0\notin\overline{\mathcal S}$. Hence, by
\eqref{S-H-representation}, there exists
$y_0\in\mathcal D\setminus\{0\}$ such that
$y_0\cdot x_0>H(y_0)$.
Set $ a_x:=y_0\cdot x_0-H(y_0)>0$. By continuity of the linear function
$x\mapsto y_0\cdot x$, there exists an open neighborhood
$\mathcal N_0$ of $x_0$ such that
\begin{equation}\label{local-H-separation}
y_0\cdot x
\geqslant
H(y_0)+\frac{a_x}{2}
\text{ for every }x\in\mathcal N_0.
\end{equation}

We next control the unbounded part of $E_\varepsilon$.
Let $\mu_0>0$ be such that
$\int_{\mathbb R^N}\mathcal J(z)e^{\mu_0|z|}\,dz<+\infty,
$
and fix $\mu_1\in(0,\mu_0)$. By compactness of
$\mathbb S^{N-1}$, there exist
$e_1,\cdots,e_m\in\mathbb S^{N-1}$ such that, for every
$\xi\in\mathbb S^{N-1}$, there exists some
$i\in\{1,\cdots,m\}$ satisfying
$\xi\cdot e_i\geqslant\frac{1}{2}$.
Set
$\tilde y_i:=\mu_1 e_i.$
Since $\mu_1<\mu_0$, we have
$\tilde y_i\in\mathcal D$ for every $i$.
Choose $R>0$ sufficiently large so that
\[
\frac{\mu_1 R}{2}
\geqslant
\max_{1\leqslant i\leqslant m}H(\tilde y_i)+1.
\]
Then, whenever $|x|\geqslant R$,  there exists $e_i$ such that
\[
\tilde y_i\cdot x=\mu_1 |x|(e_i\cdot\hat x)
\geqslant\frac{\mu_1|x|}{2}\geqslant H(\tilde y_i)+1.
\]
It follows that
\begin{equation}\label{outside-BR-cover}
\mathbb R^N\setminus B_R
\subset\bigcup_{i=1}^m
\left\{
x:\tilde y_i\cdot x\geqslant H(\tilde y_i)+1
\right\}.
\end{equation}

On the other hand,  $
K_{\varepsilon,R}
:=
E_\varepsilon\cap\overline{B_R}$
is compact. The neighborhoods
$\{\mathcal N_x\}_{x\in K_{\varepsilon,R}}$
form an open cover of $K_{\varepsilon,R}$.
Hence there exist
$x_1,\cdots,x_m\in K_{\varepsilon,R}$ and $y_1,\cdots,y_m \in\mathcal D\setminus\{0\}$ such that
$K_{\varepsilon,R}\subset
\bigcup_{i=1}^m\mathcal N_{x_i}.
$
Set  $a_i:=y_i\cdot x_i-H(y_i)>0$ and
$
\gamma_0:=\min_{1\leqslant i\leqslant m}
\frac{a_{i}}{2}>0$.
Then \eqref{local-H-separation} yields
\begin{equation}\label{inside-BR-cover}
K_{\varepsilon,R}\subset \bigcup_{i=1}^m\left\{x:
y_{i}\cdot x\geqslant H(y_{i})+\gamma_0 \right\}.
\end{equation}
Combining these two finite families,  \eqref{outside-BR-cover} and
\eqref{inside-BR-cover}, and setting
$\gamma_\varepsilon:=\min\{1,\gamma_0\}>0,$
we obtain  \eqref{finite-separation}.

We now use \eqref{finite-separation} to prove \eqref{proof-th3-2}.
Fix $\varepsilon>0$, and suppose that
$x\in\mathbb R^N$  and $t>0$ satisfy
\[
\operatorname{dist}(x,U+t\mathcal S)
\geqslant\varepsilon t.
\]
For every $y\in U$, since
$y+t\mathcal S\subset U+t\mathcal S,$
we have
$
\operatorname{dist}(x,y+t\mathcal S)\geqslant\varepsilon t$ and
\[
\operatorname{dist}
\left(
\frac{x-y}{t},\overline{\mathcal S}
\right)
=
\frac1t
\operatorname{dist}(x,y+t\overline{\mathcal S})
=
\frac1t
\operatorname{dist}(x,y+t\mathcal S)
\geqslant
\varepsilon.
\]
Therefore
\begin{equation}\label{x-U-E}
\frac{x-U}{t}\subset E_\varepsilon,
\text{ or equivalently }x-U\subset tE_\varepsilon.
\end{equation}
By \eqref{linear-representation} and \eqref{x-U-E},
\begin{equation}\label{v-Pt-E}
u(t,x)
\leqslant
v(t,x)
\leqslant
e^{f'(0)t}P_t(tE_\varepsilon).
\end{equation}
Using the finite covering
\eqref{finite-separation}, we obtain
$
tE_\varepsilon\subset\bigcup_{i=1}^m\left\{x:y_i\cdot x
\geqslant t\big(H(y_i)+\gamma_\varepsilon\big)\right\}.
$ Hence
\begin{equation}\label{Pt-union}
P_t(tE_\varepsilon)
\leqslant P_t(\bigcup_{i=1}^m\left\{x:y_i\cdot x
\geqslant t\big(H(y_i)+\gamma_\varepsilon\big)\right\})
\leqslant\sum_{i=1}^m
P_t
\left(
y_i\cdot x
\geqslant
t\big(H(y_i)+\gamma_\varepsilon\big)
\right).
\end{equation}
For each $i$, if $x_0\in\{x:y_i\cdot x
\geqslant t\big(H(y_i)+\gamma_\varepsilon\}$, then $e^{-t(H(y_i)+\gamma_\varepsilon)}e^{y_i\cdot x_0}\geqslant 1$ and \[\mathbbm 1_{\{x:y_i\cdot x
\geqslant t(H(y_i)+\gamma_\varepsilon)\}}\leqslant e^{-t(H(y_i)+\gamma_\varepsilon)}e^{y_i\cdot x} .\] The exponential inequality together with \eqref{PTE}  gives
\[
\begin{aligned}
&P_t\left( y_i\cdot x
\geqslant t(H(y_i)+\gamma_\varepsilon\big)
\right)=\int_{\mathbb R^N}
\mathbbm 1_{\{x:y_i\cdot x
\geqslant t(H(y_i)+\gamma_\varepsilon)\}}\,P_t(dx)\leqslant
e^{-t(H(y_i)+\gamma_\varepsilon)}
\int_{\mathbb R^N}
e^{y_i\cdot x}\,P_t(dx)
\\
&\quad=
e^{-t(H(y_i)+\gamma_\varepsilon)}
e^{t(\Lambda(y_i)-1)}=
e^{-t(f'(0)+\gamma_\varepsilon)},
\end{aligned}
\]
where in the last equality we used
$
H(y_i)=\Lambda(y_i)-1+f'(0).
$
Substituting this estimate into \eqref{Pt-union} and then into
\eqref{v-Pt-E}, we obtain the uniform exponential bound
\begin{equation}\label{outer-exponential-decay}
u(t,x)
\leqslant
m e^{-\gamma_\varepsilon t}
\end{equation}
whenever
$\operatorname{dist}(x,U+t\mathcal S)\geqslant\varepsilon t.$ Choose $T_{\varepsilon,\lambda}>0$ sufficiently large so that $
m e^{-\gamma_\varepsilon t}<\lambda$
for every $t\geqslant T_{\varepsilon,\lambda}$.
Then, for all $t\geqslant T_{\varepsilon,\lambda}$,
\eqref{outer-exponential-decay} implies
$$
F_\lambda(t)
\subset
\left\{
x\in\mathbb R^N:
\operatorname{dist}(x,U+t\mathcal S)
<
\varepsilon t
\right\}.
$$
Consequently,
$ \sup_{x\in F_\lambda(t)} \operatorname{dist}(x,U+t\mathcal S)\leqslant
\varepsilon t$
for all sufficiently large $t$.
Therefore
$$
\limsup_{t\to+\infty}
\frac1t
\sup_{x\in F_\lambda(t)}
\operatorname{dist}(x,U+t\mathcal S)
\leqslant
\varepsilon.
$$
Since $\varepsilon>0$ is arbitrary, we conclude  \eqref{proof-th3-2}.
\end{proof}

\begin{proof}[Proof of Theorem~\ref{th2}]
Recall that
$\mathcal O=\mathcal S-\mathbf p$,  $\mathbb R^+=[0,+\infty)$ and $\mathcal W=\mathbb{R}^{+}\mathcal{U}(U)+\mathcal S$. Then
$\mathcal W-\mathbf p =\mathbb{R}^{+}\mathcal{U}(U)+\mathcal O$. We further show that
$\mathcal W=\operatorname{int}\overline{\mathcal W}$.
Assume first that $\mathcal U(U)\neq\varnothing$,
since $\mathcal U(U)$ is compact and contained in
$\mathbb S^{N-1}$, the cone $\mathbb R^+\mathcal U(U)$ is closed. Recall that
$\mathcal W-\mathbf p=\mathbb R^+\mathcal U(U)+\mathcal O$,
where $\mathcal O$ is bounded, open and convex, with
$0\in\mathcal O$. Let $z\in\overline{\mathbb R^+\mathcal U(U)+\mathcal O}$ and
$\theta\in[0,1)$. Choose
$a_n\in \mathbb R^+\mathcal U(U)$ and $q_n\in\mathcal O$ such that
$a_n+q_n\to z$. Since $\mathcal O$ is bounded, up to a
subsequence,
\[
a_n\to a\in \mathbb R^+\mathcal U(U),\
q_n\to q\in\overline{\mathcal O},
\]
and hence $z=a+q$. Since $\mathbb R^+\mathcal U(U)$ is a cone and $\theta\overline{\mathcal O}\subset\mathcal O$ for
$\theta\in[0,1)$, we have
\[
\theta z=\theta a+\theta q\in \mathbb R^+\mathcal U(U)+\mathcal O.
\]
Therefore, if
$x\in\operatorname{int}\overline{\mathbb R^+\mathcal U(U)+\mathcal O}$, then, choosing
$\rho>1$ sufficiently close to $1$ so that
$\rho x\in\overline{\mathbb R^+\mathcal U(U)+\mathcal O}$ and setting
$\theta=\rho^{-1}\in(0,1)$, we obtain
$x=\theta(\rho x)\in \mathbb R^+\mathcal U(U)+\mathcal O$.
Thus
\[
\operatorname{int}\overline{\mathbb R^+\mathcal U(U)+\mathcal O}
\subset \mathbb R^+\mathcal U(U)+\mathcal O.
\]
The reverse inclusion follows from the openness of
$\mathbb R^+\mathcal U(U)+\mathcal O$. Hence
\[
\mathbb R^+\mathcal U(U)+\mathcal O
=
\operatorname{int}\overline{\mathbb R^+\mathcal U(U)+\mathcal O},
\]
and therefore $
\mathcal W=\operatorname{int}\overline{\mathcal W}$.

If $\mathcal U(U)=\varnothing$, then
$\mathcal W=\mathcal S$, and the same conclusion follows directly
from the fact that $\mathcal S$ is open and convex.

 Since $\mathbb{R}^{+}\mathcal{U}(U)$ is star-shaped with respect to the origin and $\mathcal O$ is convex with $0\in\mathcal O$, for every $\theta\in[0,1]$,
$
\theta(\mathcal W-\mathbf p)=\theta(\mathbb{R}^{+}\mathcal{U}(U)+\mathcal O)\subset
\mathbb{R}^{+}\mathcal{U}(U)+\mathcal O$.
Thus $\mathcal W$ is star-shaped with respect to $\mathbf p$.
Moreover, $\mathcal W$ is open since $\mathcal O$ is open.

For $e\in\mathbb S^{N-1}$, define
$w(e):=\sup\{r\geqslant0:\mathbf p+re\in\mathcal W\}.
$
The star-shapedness and openness of $\mathcal W$ imply
$
\mathcal W\cap(\mathbf p+\mathbb R^+e)=
\{\mathbf p+re:0\leqslant r<w(e)\}.
$
Consequently,
\begin{equation}\label{W-radial-representation}
\mathcal W
=\left\{\mathbf p+re:e\in\mathbb S^{N-1},-w(-e)<r<w(e)
\right\}.
\end{equation}
Since $\mathcal S\subset\mathcal W$, we also have
$w(e)\geqslant\sigma(e)$
for every $e\in\mathbb S^{N-1}$.

Assume first that $\mathcal U(U)\neq\varnothing$. For
$\xi\in\mathcal U(U)$, set
$\mathcal A_\xi:=\mathcal O+\mathbb R^+\xi$
and
$w_\xi(e):=\sup\{r\geqslant0:re\in\mathcal A_\xi\}$.
Then $\mathcal W-\mathbf p=
\bigcup_{\xi\in\mathcal U(U)}\mathcal A_\xi$,
and hence $w(e)=\sup_{\xi\in\mathcal U(U)}w_\xi(e)$. We claim that
\begin{equation}\label{Axi-halfspace}
\overline{\mathcal A_\xi}
=
\bigcap_{
\eta\cdot\xi\leqslant0}
\left\{
x\in\mathbb R^N:
x\cdot\eta
\leqslant
c_\eta^*-\mathbf p\cdot\eta,\ \eta\in \mathbb S^{N-1}
\right\}.
\end{equation}
Indeed, since $\overline{\mathcal O}$ is compact,
$\overline{\mathcal A_\xi}=\overline{\mathcal O}+\mathbb R^+\xi.$
Let $x=z+r\xi\in\overline{\mathcal A_\xi}$, with
$z\in\overline{\mathcal O}$ and $r\geqslant0$. If
$\eta\cdot\xi\leqslant0$, then
$
x\cdot\eta=z\cdot\eta+r\xi\cdot\eta\leqslant
z\cdot\eta\leqslant c_\eta^*-\mathbf p\cdot\eta,
$
where we used
$
\sup_{z\in\overline{\mathcal O}}z\cdot\eta
=c_\eta^*-\mathbf p\cdot\eta$ by Hypothesis \ref{mainhypothesis}.
This proves one inclusion in \eqref{Axi-halfspace}. Conversely, suppose that $x_0$ belongs to the right-hand side
of \eqref{Axi-halfspace} but
$x_0\notin\overline{\mathcal A_\xi}$. Since
$\overline{\mathcal A_\xi}$ is a nonempty closed convex set,
the separation theorem yields
$\eta_0\in\mathbb S^{N-1}$ such that
$$
x_0\cdot\eta_0
>
\sup_{x\in\overline{\mathcal A_\xi}}x\cdot\eta_0.
$$
We necessarily have $\eta_0\cdot\xi\leqslant0$. Indeed, if
$\eta_0\cdot\xi>0$, then, for any fixed
$z\in\overline{\mathcal O}$,
$
(z+r\xi)\cdot\eta_0\longrightarrow+\infty$
as $r\to+\infty$, which contradicts the preceding strict separation. Since $\eta_0\cdot\xi\leqslant0$, by Hypothesis \ref{mainhypothesis}, we have
\[
\sup_{x\in\overline{\mathcal A_\xi}}x\cdot\eta_0
=
\sup_{z\in\overline{\mathcal O},r\geqslant0}
(z+r\xi)\cdot\eta_0=
\sup_{z\in\overline{\mathcal O}}z\cdot\eta_0=
c_{\eta_0}^*-\mathbf p\cdot\eta_0.
\]
Hence
$x_0\cdot\eta_0>c_{\eta_0}^*-\mathbf p\cdot\eta_0,
$
contradicting the fact that $x_0$ belongs to the right-hand
side of \eqref{Axi-halfspace}. Thus \eqref{Axi-halfspace}
holds.

Because
$c_\eta^*-\mathbf p\cdot\eta>0$, substituting $z=re$ shows
that only the directions satisfying $\eta\cdot e>0$ impose an
upper bound on $r$. Therefore
$$
w_\xi(e)
=
\inf_{
\eta\cdot e>0,\,
\eta\cdot\xi\leqslant0}
\frac{c_\eta^*-\mathbf p\cdot\eta}
{\eta\cdot e},
$$
with the convention $\inf\varnothing=+\infty$. Consequently, we obtain \eqref{th1-2}:
\[w(e)=\sup_{\xi\in\mathcal U(U)}\inf_{
\eta \cdot e>0,\,\eta\cdot\xi\leqslant0}\frac{c_\eta^*-\mathbf p\cdot\eta} {\eta\cdot e}.
\]
If $\mathcal U(U)=\varnothing$, then
 $\mathcal W=\mathcal S$.
In this case we set $w(e)=\sigma(e)$.

We next prove the interior spreading property.
Let $C\subset\mathcal W$ be compact. For each $x\in C$,
choose a representation
$x=\tau_x\xi_x+s_x,$
where $\tau_x\geqslant0$, $s_x\in\mathcal S$, and, if
$\tau_x>0$, $\xi_x\in\mathcal U(U)$.

Choose $\rho_x>0$ such that
$K_x:=\overline{B_{2\rho_x}(s_x)}
\subset\mathcal S$.
If $\tau_x>0$, since
$\xi_x\in\mathcal U(U_\delta)$, there exist
$y_t^x\in U_\delta$ such that
$\frac{|t\tau_x\xi_x-y_t^x|}{t}\to0$.
Set
$h_t^x:=\frac{t\tau_x\xi_x-y_t^x}{t}$.
For all sufficiently large $t$,
$B_{\rho_x}(s_x)+h_t^x\subset K_x$.
Hence, for every
$z=\tau_x\xi_x+s'\in
\tau_x\xi_x+B_{\rho_x}(s_x)$,
$tz=y_t^x+t(s'+h_t^x) \in U_\delta+tK_x$.
Lemma~\ref{Udelta+ts} therefore gives
$\inf_{z\in\tau_x\xi_x+B_{\rho_x}(s_x)}u(t,tz)\to1$ as $t\to+\infty$. When $\tau_x=0$, the same conclusion follows by fixing
$y_0\in U_\delta$ and using $h_t=-y_0/t$. The sets
$N_x:=\tau_x\xi_x+B_{\rho_x}(s_x)$
form an open cover of $C$. By compactness, a finite subcover
exists. Consequently,
\[
\lim_{t\to+\infty}\inf_{x\in C}u(t,tx)=1.
\]

Finally, we prove the exterior spreading property. Let
$C\subset\mathbb R^N\setminus\overline{\mathcal W}$ be a
compact set. For every $z\in C$, since
$\mathbf p\in\mathcal W$, we have $z\neq\mathbf p$ and may
write
$z=\mathbf p+r_z e_z$,
$r_z:=|z-\mathbf p|>0$ and
$e_z:=\frac{z-\mathbf p}{|z-\mathbf p|}
\in\mathbb S^{N-1}$.
By the radial representation
\eqref{W-radial-representation} and the fact that
$z\notin\overline{\mathcal W}$, we have
$r_z>w(e_z)$.
In particular, $w(e_z)<+\infty$. Hence
Lemma~\ref{w>w(e)}, applied with $e=e_z$ and $w=r_z$,
provides an open neighborhood $\mathcal N_z$ of $z$ such that
$\lim_{t\to+\infty}\sup_{y\in\mathcal N_z}u(t,ty)=0$.
Shrinking $\mathcal N_z$ if necessary, we may also assume that $\mathcal N_z \subset\mathbb R^N\setminus\overline{\mathcal W}$.

The family $\{\mathcal N_z\}_{z\in C}$ is an open cover of
$C$. Since $C$ is compact, there exist
$z_1,\cdots,z_m\in C$ such that
$C\subset\bigcup_{i=1}^m\mathcal N_{z_i}$.
Therefore,
\[
\sup_{y\in C}u(t,ty) \leqslant\max_{1\leqslant i\leqslant m}\sup_{y\in\mathcal N_{z_i}}u(t,ty).
\]
Since the number of neighborhoods is finite and
\[
\sup_{y\in\mathcal N_{z_i}}u(t,ty)\longrightarrow0
\text{ as }t\to+\infty
\]
for every $i=1,\cdots,m$, we conclude that
$\lim_{t\to+\infty}\sup_{y\in C}u(t,ty)=0.$
This proves the exterior spreading property in
\eqref{definitionw}.

It remains to prove \eqref{th2-3}. Fix $\lambda\in(0,1)$ and a compact set $\mathcal K\subset\mathbb R^N$ satisfying
$\overline{\mathcal K\cap\mathcal W}
=\mathcal K\cap\overline{\mathcal W}$,
and set
$A_t:=K\cap\frac1tF_\lambda(t)$.

We first show that
\[
\sup_{x\in A_t}\operatorname{dist}
(x,\mathcal K\cap\mathcal W)\to0
\text{ as }t\to+\infty.
\]
Suppose otherwise. Then there exist $\varepsilon>0$,
$t_n\to+\infty$, and $x_n\in A_{t_n}$ such that
$\operatorname{dist}
(x_n,\mathcal K\cap\mathcal W)\geqslant\varepsilon$.
Since $\mathcal K$ is compact, up to a subsequence,
$x_n\to x\in \mathcal K$. We claim that $x\in\overline{\mathcal W}$.
Otherwise, there exists $r>0$ such that
$\overline{B_r(x)}
\subset\mathbb R^N\setminus\overline{\mathcal W}$.
For all sufficiently large $n$, $x_n\in B_r(x)$, and the
exterior spreading property gives
\[
u(t_n,t_nx_n)\to 0,
\]
which contradicts $x_n\in A_{t_n}$, namely
$u(t_n,t_nx_n)>\lambda$. Hence
$x\in  \mathcal K\cap\overline{\mathcal W}
= \overline{\mathcal K\cap\mathcal W}$,
and therefore
$\operatorname{dist}
\bigl(x_n,\mathcal K\cap\mathcal W\bigr)\to0$,
again a contradiction.

We next prove the reverse estimate. Let $\varepsilon>0$. Since
$\overline{\mathcal K\cap\mathcal W}=\mathcal K\cap\overline{\mathcal W}$
is compact and $\mathcal K\cap\mathcal W$ is dense in it, there exist $x_1,\ldots,x_m\in \mathcal K\cap\mathcal W$ such that
$\mathcal K\cap\overline{\mathcal W}
\subset
\bigcup_{i=1}^m B_\varepsilon(x_i)$.
The finite set $\{x_1,\ldots,x_m\}$ is a compact subset of
$\mathcal W$. Hence, by the interior spreading property,
\[
\min_{1\leqslant i\leqslant m}
u(t,tx_i)\to 1
\text{ as }t\to+\infty.
\]
Thus, for all sufficiently large $t$,
$x_i\in A_t$ for every $i=1,\ldots,m$. Consequently, for every
$x\in \mathcal K\cap\mathcal W$,
\[
\sup_{x\in\mathcal  K\cap\mathcal W}
\operatorname{dist}(x,A_t)
\to0.
\]

If $\mathcal K\cap\mathcal W=\varnothing$, then the assumption $\overline{\mathcal K\cap\mathcal W}
=\mathcal K\cap\overline{\mathcal W}$ gives
$\mathcal K\cap\overline{\mathcal W}=\varnothing$, and the exterior spreading property implies $A_t=\varnothing$ for all sufficiently large $t$. Therefore, in all cases,
\[
d_{\mathcal H}\left(\mathcal K\cap\frac1tF_\lambda(t), \mathcal K\cap\mathcal W\right)\to0\text{ as }t\to+\infty.
\]
This proves \eqref{th2-3} and completes the proof.
\end{proof}

\begin{proof}[Proof of Theorem~\ref{th1}]
The first limit in~\eqref{th1-1} follows directly from the
interior spreading property of $\mathcal W$. Indeed, if
$r\in(-w(-e),w(e))$, then $\mathbf p+re\in\mathcal W$ by the radial representation
\eqref{th2-2}.

For the second limit, let
$r\notin[-w(-e),w(e)]$. If $r>w(e)$, then necessarily $w(e)<+\infty$, and Lemma~\ref{w>w(e)}, applied in the direction $e$ with $w=r$,
gives
$ u(t,t(\mathbf p+re))\to0$
as $t\to+\infty$. If instead $r<-w(-e)$, then $-r>w(-e)$, so necessarily $w(-e)<+\infty$. Applying Lemma~\ref{w>w(e)} in the direction $-e$ with $w=-r$ yields again
$u(t,t(\mathbf p+re))\to0$
as $t\to+\infty$. This proves~\eqref{th1-1}.
\end{proof}
\begin{proof}[Proof of Proposition \ref{finite-CP}]
Let $l$ be a global Lipschitz constant of the zero extension of $f$, choose
$\lambda>l$, and  $D:=\{x\cdot e< K\}$. Set
$U(t,x):=e^{-\lambda(t-T_1)}\underline v(t,x)$ and
$V(t,x):=e^{-\lambda(t-T_1)}\overline v(t,x)$.
Then $U$ is bounded and upper semicontinuous, while $V$ is
bounded and lower semicontinuous. Define
$$
G(t,r):=
\lambda r-e^{-\lambda(t-T_1)}
f\bigl(e^{\lambda(t-T_1)}r\bigr).
$$
We first verify that $U$ and $V$ are respectively a viscosity
subsolution and a viscosity supersolution of
\begin{equation}\label{finite-transformed}
W_t-\mathcal J\star W+W-\zeta\cdot\nabla W+G(t,W)=0.
\end{equation}
Suppose that $\varphi$ touches $U$ from above at
$(t_0,x_0)$. Then
$\psi(t,x):=e^{\lambda(t-T_1)}\varphi(t,x)$ touches
$\underline v$ from above at the same point. Hence
$\psi_t-\mathcal J\star\underline v+\underline v
-\zeta\cdot\nabla\psi-f(\underline v)\leqslant0$.
Using
$\psi_t=e^{\lambda(t-T_1)}(\varphi_t+\lambda\varphi)$,
$\nabla\psi=e^{\lambda(t-T_1)}\nabla\varphi$, and
$\mathcal J\star\underline v
=e^{\lambda(t-T_1)}\mathcal J\star U$, we obtain
\[
\varphi_t-\mathcal J\star U+U-\zeta\cdot\nabla\varphi
+G(t_0,U(t_0,x_0))\leqslant0.
\]
Thus $U$ is a viscosity subsolution of
\eqref{finite-transformed}. The same argument, with a test
function touching from below, shows that $V$ is a viscosity
supersolution of
\eqref{finite-transformed}.

If $r\geqslant s$, the Lipschitz continuity of $f$ yields
$G(t,r)-G(t,s)\geqslant(\lambda-l)(r-s)$. Set
$\kappa:=\lambda-l>0$. Hence
\begin{equation}\label{finite-proper}
G(t,r)-G(t,s)\geqslant\kappa(r-s),\ r\geqslant s.
\end{equation}

Suppose by contradiction that
$m:=\sup_{[T_1,T_2]\times\mathbb R^N}(U-V)>0$.
Since $U\leqslant V$ on $\{x\cdot e\geqslant K\}$, a positive
maximum can only occur in $[T_1,T_2]\times D$.
Choose $(t^*,x^*)$ such that
$U(t^*,x^*)-V(t^*,x^*)>3m/4$.

Let $\chi(x):=\sqrt{1+|x|^2}$ and $m_1:=\int_{\mathbb R^N}|x|\mathcal J(x)dx<+\infty$ by (J). Then we have
$|\nabla\chi|\leqslant1$ and
$|\chi(x-z)-\chi(x)|\leqslant|z|$. Choose $\iota>0$ such that $2\iota\chi(x^*)<m/4$ and $2\iota(m_1+|\zeta|)<\kappa m/4$.

For $\rho,\epsilon>0$, set
$$
\Phi_{\rho,\epsilon}(t,s,x,y)
:=
U(t,x)-V(s,y)
-\frac{|x-y|^2}{2\rho}
-\frac{|t-s|^2}{2\epsilon}
-\iota\chi(x)-\iota\chi(y).
$$
Since $U,V$ are bounded and $\chi(x)\to+\infty$ as
$|x|\to+\infty$, $\Phi_{\rho,\epsilon}$ attains its maximum.
Let $(t_\rho,s_\rho,x_\rho,y_\rho)$ be a maximum point.
Evaluating $\Phi_{\rho,\epsilon}$ at
$(t^*,t^*,x^*,x^*)$ gives
\begin{equation}\label{finite-positive}
U(t_\rho,x_\rho)-V(s_\rho,y_\rho)>\frac m2.
\end{equation}
The standard penalization argument gives
$|t_\rho-s_\rho|+|x_\rho-y_\rho|\to0$ as
$\rho,\epsilon\to0$\cite{AT1996,BCI2008,BI2008}. For fixed $\iota$, the points
$x_\rho,y_\rho$ remain bounded. Passing to a subsequence, write
$t_\rho,s_\rho\to\widehat t$ and
$x_\rho,y_\rho\to\widehat x$. We have $\widehat t>T_1$ and $\widehat x\in D$. Indeed, if
$\widehat t=T_1$, semicontinuity and the initial data give
$$
\frac m2
\leqslant
\liminf_{\rho,\epsilon\to0}
\bigl(U(t_\rho,x_\rho)-V(s_\rho,y_\rho)\bigr)
\leqslant
U(T_1,\widehat x)-V(T_1,\widehat x)
\leqslant0,
$$
a contradiction. If $\widehat x\in D^c$, the same argument
using \eqref{finite-exterior} again contradicts
\eqref{finite-positive}. Since $D$ is open, for sufficiently
small $\rho,\epsilon$ both $(t_\rho,x_\rho)$ and
$(s_\rho,y_\rho)$ belong to $(T_1,T_2]\times D$.

Fix now $(s_\rho,y_\rho)$. By the maximality of
$\Phi_{\rho,\epsilon}$, $(t_\rho,x_\rho)$ maximizes
\[
U(t,x)-\left(
\frac{|x-y_\rho|^2}{2\rho}
+\frac{|t-s_\rho|^2}{2\epsilon}
+\iota\chi(x)
\right).
\]
Therefore, after adding a suitable constant $C_\rho$,
\[
\varphi_\rho^U(t,x):=\frac{|x-y_\rho|^2}{2\rho}
+\frac{|t-s_\rho|^2}{2\epsilon}+\iota\chi(x)+C_\rho
\]
touches $U$ from above at $(t_\rho,x_\rho)$.
Thus, with
$a_\rho:=(t_\rho-s_\rho)/\epsilon$ and
$p_\rho^U:=(x_\rho-y_\rho)/\rho
+\iota\nabla\chi(x_\rho)$, the viscosity subsolution property
gives
\begin{equation}\label{finite-sub}
a_\rho-(\mathcal J\star U)(t_\rho,x_\rho)
+U(t_\rho,x_\rho)-\zeta\cdot p_\rho^U
+G(t_\rho,U(t_\rho,x_\rho))
\leqslant0.
\end{equation}
Similarly, fixing $(t_\rho,x_\rho)$, the function
\[
\varphi_\rho^V(s,y):=
-\frac{|x_\rho-y|^2}{2\rho}
-\frac{|t_\rho-s|^2}{2\epsilon}
-\iota\chi(y)+\widetilde C_\rho
\]
touches $V$ from below at $(s_\rho,y_\rho)$ with constant $\widetilde C_\rho$. Therefore, with
$p_\rho^V:=(x_\rho-y_\rho)/\rho
-\iota\nabla\chi(y_\rho)$,
\begin{equation}\label{finite-super}
a_\rho-(\mathcal J\star V)(s_\rho,y_\rho)
+V(s_\rho,y_\rho)-\zeta\cdot p_\rho^V
+G(s_\rho,V(s_\rho,y_\rho))
\geqslant0.
\end{equation}

We now estimate the nonlocal terms. By maximality, for every
$z\in\mathbb R^N$,
\[\Phi_{\rho,\epsilon}(t_\rho,s_\rho,x_\rho-z,y_\rho-z)
\leqslant
\Phi_{\rho,\epsilon}(t_\rho,s_\rho,x_\rho,y_\rho).\]
Since the  translation
$(x_\rho,y_\rho)\mapsto(x_\rho-z,y_\rho-z)$ leaves
$|x_\rho-y_\rho|^2$ unchanged, we obtain
\[
U(t_\rho,x_\rho-z)-V(s_\rho,y_\rho-z)
\leqslant
U(t_\rho,x_\rho)-V(s_\rho,y_\rho)+2\iota|z|.
\]
Multiplying by $\mathcal J(z)\geqslant0$ and integrating gives\[
\begin{aligned}
(\mathcal J\star U)(t_\rho,x_\rho)
-(\mathcal J\star V)(s_\rho,y_\rho)
\leqslant
U(t_\rho,x_\rho)-V(s_\rho,y_\rho)+2\iota m_1.
\end{aligned}\]
Moreover,
$-\zeta\cdot(p_\rho^U-p_\rho^V)\geqslant-2\iota|\zeta|$.
By \eqref{finite-positive} and \eqref{finite-proper},
$$
G(t_\rho,U(t_\rho,x_\rho))
-G(t_\rho,V(s_\rho,y_\rho))
>\frac{\kappa m}{2}.
$$
Since $G$ is uniformly continuous on bounded sets,
\[
G(t_\rho,U(t_\rho,x_\rho))-G(s_\rho,V(s_\rho,y_\rho))
\geqslant \frac{\kappa m}{2} -\omega(|t_\rho-s_\rho|)
\]
for some modulus of continuity $\omega$. Subtracting \eqref{finite-super} from \eqref{finite-sub} and
using the previous estimates yields
\[
0\geqslant
-2\iota(m_1+|\zeta|)
+\frac{\kappa m}{2}
-\omega(|t_\rho-s_\rho|).
\]
Letting $\rho,\epsilon\to0$, we obtain
$0\geqslant-2\iota(m_1+|\zeta|)+\kappa m/2$.
By our choice of $\iota$, the right-hand side is larger than
$\kappa m/4>0$, a contradiction. Thus $U\leqslant V$, and since the exponential factor is
positive, $\underline v\leqslant\overline v$ on
$[T_1,T_2]\times\mathbb R^N$.
\end{proof}

\section*{Acknowledgments}
\noindent

The authors would like to thank Prof. Luca Rossi (Sapienza University of Rome), Dr. Lele Du (Sapienza University of Rome), and  Dr. Teng-Long Cui (Lanzhou University) for their constructive comments and helpful suggestions.  H. Guo was supported by the fundamental research funds for the central universities and NSF of China (No. 12471201).  W.-T. Li was partially supported by NSF of China (12531008;12271226).

\end{document}